\documentclass[reqno,11pt]{amsart}
\usepackage{amsmath, latexsym, amsfonts, amssymb, amsthm, amscd, stmaryrd}
\usepackage{graphics,epsf,psfrag,epsfig}
\usepackage[colorlinks=true, pdfstartview=FitV, linkcolor=blue, citecolor=blue, urlcolor=blue,pagebackref=false,hyperindex,breaklinks]{hyperref}

\numberwithin{equation}{section}

\newtheorem{theorem}{Theorem}[section]
\newtheorem{lemma}[theorem]{Lemma}
\newtheorem{proposition}[theorem]{Proposition}

\newtheorem{rem}[theorem]{Remark}

\renewcommand{\ge}{\geq}
\renewcommand{\le}{\leq}
\newcommand{\ind}{\mathbf{1}}

\renewcommand{\tilde}{\widetilde}

\newcommand{\grad}{\nabla}

\newcommand{\cA}{{\ensuremath{\mathcal A}} }
\newcommand{\cF}{{\ensuremath{\mathcal F}} }

\newcommand{\cG}{{\ensuremath{\mathcal G}} }
\newcommand{\cE}{{\ensuremath{\mathcal E}} }
\newcommand{\cH}{{\ensuremath{\mathcal H}} }
\newcommand{\cC}{{\ensuremath{\mathcal C}} }

\newcommand{\cL}{{\ensuremath{\mathcal L}} }
\newcommand{\cT}{{\ensuremath{\mathcal T}} }

\newcommand{\cB}{{\ensuremath{\mathcal B}} }

\DeclareMathSymbol{\leqslant}{\mathalpha}{AMSa}{"36} 
\DeclareMathSymbol{\geqslant}{\mathalpha}{AMSa}{"3E} 
\DeclareMathSymbol{\eset}{\mathalpha}{AMSb}{"3F}     
\renewcommand{\leq}{\;\leqslant\;}                   
\renewcommand{\geq}{\;\geqslant\;}                   
\newcommand{\dd}{\,\text{\rm d}}             

\newcommand{\sumtwo}[2]{\sum_{\substack{#1 \\ #2}}} 

\newcommand{\bbE}{{\ensuremath{\mathbb E}} }

\newcommand{\bbN}{{\ensuremath{\mathbb N}} }

\newcommand{\bbP}{{\ensuremath{\mathbb P}} }

\newcommand{\bbR}{{\ensuremath{\mathbb R}} }

\newcommand{\bbZ}{{\ensuremath{\mathbb Z}} }
\newcommand{\kA}{{\ensuremath{\mathfrak A}} }

\newcommand{\gb}{\beta}

\newcommand{\gep}{\varepsilon}       

\newcommand{\gD}{\Delta}

\newcommand{\go}{\omega}

\newcommand{\gO}{\Omega}
\newcommand{\gl}{\lambda}

\makeatletter
\def\captionfont@{\footnotesize}
\def\captionheadfont@{\scshape}

\long\def\@makecaption#1#2{%
  \vspace{2mm}
  \setbox\@tempboxa\vbox{\color@setgroup
    \advance\hsize-6pc\noindent
    \captionfont@\captionheadfont@#1\@xp\@ifnotempty\@xp
        {\@cdr#2\@nil}{.\captionfont@\upshape\enspace#2}%
    \unskip\kern-6pc\par
    \global\setbox\@ne\lastbox\color@endgroup}%
  \ifhbox\@ne 
    \setbox\@ne\hbox{\unhbox\@ne\unskip\unskip\unpenalty\unkern}%
  \fi
  \ifdim\wd\@tempboxa=\z@ 
    \setbox\@ne\hbox to\columnwidth{\hss\kern-6pc\box\@ne\hss}%
  \else 
    \setbox\@ne\vbox{\unvbox\@tempboxa\parskip\z@skip
        \noindent\unhbox\@ne\advance\hsize-6pc\par}%
\fi
  \ifnum\@tempcnta<64 
    \addvspace\abovecaptionskip
    \moveright 3pc\box\@ne
  \else 
    \moveright 3pc\box\@ne
    \nobreak
    \vskip\belowcaptionskip
  \fi
\relax
}
\makeatother
\def\writefig#1 #2 #3 {\rlap{\kern #1 truecm
\raise #2 truecm \hbox{#3}}}

\newcommand{\bcos}{{\overline \cos}}
\newcommand{\var}{{\rm Var}}
\newcommand{\Tm}{T_{\rm mix}}

\title[Exclusion on the circle]
{The Simple Exclusion Process on the Circle has a diffusive Cutoff Window}

\address{IMPA, Estrada Dona Castorina 110,
RJ-22460-320 Rio de Janeiro, Brasil}
\email{lacoin@impa.br}
\author{Hubert Lacoin}

\begin{document}

\begin{abstract}
In this paper, we investigate the mixing time of the simple exclusion process on the circle with $N$ sites, with a number of particle $k(N)$ 
tending to infinity, both from the worst initial condition and from a typical initial condition. 
We show that the worst-case mixing time is asymptotically equivalent to $(8\pi^2)^{-1}N^2\log k$, 
while the cutoff window, is identified to be $N^2$.
Starting from a typical condition, we show that there is no cutoff and that the mixing time is of order $N^2$. 
\end{abstract}

\maketitle

 \tableofcontents

\section{Introduction}

The Symmetric Simple Exclusion Process (to which we sometimes refer simply as \textit{the  Simple Exclusion}) is one of the simplest particle system with local 
interactions. It can be considered as a toy model for the  relaxation of a gas of particle and was introduced by Spitzer in \cite{cf:Spitzer}. 
Since then, it has been the object of
a large number of studies by mathematicians and theoretical physicists, who investigated many of its properties such as
the evolution rules for the particle density and tried to derive Fick's law from microscopic dynamics,
studied to motion of an individual tagged particle (see \cite{cf:Liggett2, cf:KL} for reviews on the subject and references therein).
More recently \cite{cf:DSC0, cf:DSC, cf:Lac, cf:LeeYau, cf:Quas, cf:Yau} an interest has been developed for the convergence to equilibrium of the process on a finite graph in terms of mixing time, which is the object of our study.

\subsection{The Process}

We  consider $\bbZ_N:=\bbZ/N\bbZ$, the discrete circle with $N$ sites and place $k\in\{1,\dots,N-1\}$ particles
on it, with \textit{at most} one particle per site. With a slight abuse of notation, we  sometimes use elements of $\{1,\dots,N\} \subset \bbZ$ to refer to elements of $\bbZ_N$.

\medskip

The Simple Exclusion on $\bbZ_N$ is a dynamical evolution of the particle system which can be described informally as follows: 
each particle tries to jump independently on its neighbors with transition rates $p(x,x+1)=p(x,x-1)=1$, but the jumps are cancelled if
a particle tries to jump on a site which is already occupied (see Figure \ref{partisys} in Section \ref{fluctuat} for a graphical representation).

\medskip

More formally, our state-space is defined by
\begin{equation}
\gO= \gO_{N,k}= \left\{ \eta \in \{0,1\}^{\bbZ_N}\ | \ \sum_{x=1}^N \eta(x)= k \right\}.
\end{equation}
Given $\eta\in \gO$ define $\eta^x$ the configuration obtained by exchanging the content of site $x$ and $x+1$
\begin{equation}
\begin{cases}
\eta^x(x):=\eta(x+1),\\
\eta^x(x+1):=\eta(x),\\
\eta^x(y)=\eta(y), \quad \forall y\notin\{x,x+1\}.
\end{cases}
\end{equation}

The exclusion process on $\bbZ_N$ with $k$ particle is the continuous time Markov process on $\gO_{N,k}$ whose generator is given by
\begin{equation}\label{crading}
 (\mathcal L f)(\eta):=\sum_{x\in \bbZ_N} f(\eta^x)-f(\eta).
\end{equation}
The unique probability measure left invariant by $\mathcal L$ is the uniform probability measure on $\gO$ which we denote by $\mu$.
Given $\chi \in \gO,$ we let $(\eta^{\chi}_t)_{t\ge 0}$ denote the trajectory of the Markov chain starting from $\chi$.

\medskip

We want to know how long we must wait to reach the equilibrium state of the particle system, for which all configurations are equally likely.
We measure the distance to equilibrium is measured in terms of total variation distance.
If $\alpha$ and $\beta$ are two probability measures on $\gO$, the total variation distance between $\alpha$ and $\gb$ is defined to be
  \begin{equation}\label{tv}
  \| \alpha -\beta\|_{TV}:=\frac{1}{2}\sum_{\go\in \gO} |\alpha(\go)-\beta(\go)|=\sum_{\go\in \gO} (\alpha(\go)-\beta(\go))_+,
 \end{equation}
 where $x_+=\max(x,0)$ is the positive part of $x$. It measures how well one can couple two variables with law $\alpha$ and $\beta$.
 We define the worst-case distance to equilibrium  at time $t>0$ as follows
\begin{equation}\label{dat}
d(t)=d^{N,k}(t):=\max_{\chi\in \gO_{N,k}} \| P^\chi_t-\mu\|_{TV}.
\end{equation}
Similarly we define the typical distance from equilibrium at time $t>0$ as
\begin{equation}
 {\bf d} (t)={\bf d}^{N,k}(t):= \frac{1}{\# \gO_{N,k}} \sum_{\chi\in \gO_{N,k}} \| P^\chi_t-\mu\|_{TV}.
\end{equation}
For a given $\gep>0$ we define the $\gep$-mixing-time (or simply the mixing time when $\gep=1/4$) 
to be the time needed for the system to be at distance $\gep$ from equilibrium 
\begin{equation}
 \Tm^{N,k}(\gep):=\inf\{t\ge 0\  | \ d^{N,k}(t)\le \gep\}.
\end{equation}
Let us mention that the convergence to equilibrium has also been studied in terms of asymptotic rates: 
it has been known for a long time (see e.g.\ \cite[Corollary 12.6]{cf:LPW}) that for any reversible Markov chain
\begin{equation}
\lim_{t\to \infty}t^{-1}\log d(t) =  \lim_{t\to \infty} t^{-1}{\bf d} (t)=-\gl_1.
\end{equation}
exists and that $\gl_1>0$ is the smallest nonzero eigenvalue of $-\mathcal L$,
usually referred to as the spectral gap. 
 Note that the knowledge of the spectral-gap also give an information on $d(t)$ for finite $t$, as we have (cf.\ \cite{cf:LPW}[Theorem 12.3])
\begin{equation}\label{sgapin}
 \frac{1}{2} e^{-\gl_1 t} \le d(t)\le |\gO|^{-1}e^{-\gl_1 t}.
\end{equation}

\medskip

The exclusion process can in fact be defined on an arbitrary graph and its mixing property have been the object of a large number of works.
Let us mention a few of them here.
Let us start with the mean-field case: in \cite{cf:DS2}, the study of the exclusion on the complete graph with $N/2$ particles is reduced to the study of the birth and death chain and a sharp asymptotic 
for the mixing time is given
using a purely algebraic approach (see also \cite{cf:LL} for a probabilistic approach of the problem for arbitrary $k$).

\medskip

The problem on the lattice is much more delicate. Let us mention a few results that were obtained on the torus $(\bbZ_N)^d$:
in \cite{cf:Quas} (and  also independently in \cite{cf:DSC0}), comparisons with the mean-field model were used 
to prove that  $\gl_1:=O(N^{-2})$, and thus via \eqref{sgapin} that there exists a constant $C_d$ such that
\begin{equation}\label{gapbound}
 \Tm\le C_d N^2 \log \binom{N^d}{k}.
\end{equation}

\medskip

In \cite{cf:DSC, cf:LeeYau, cf:Yau}, the related problem of the  \textit{log-Sobolev constant} for the process was studied. 
In particular, in \cite{cf:Yau}, a sharp bound (up to multiplicative constant) on the log-Sobolev constant was proved for 
the exclusion process on the grid which allowed to improve \eqref{gapbound} into
\begin{equation}
  \Tm\le C_d N^2 \log\log \binom{N^d}{k}.
\end{equation}
In \cite{cf:Mor}, using the \textit{chameleon process}, this upper bound is improved in the case of small $k$ by showing that 
\begin{equation}
   \Tm\le C (\log d) N^2 \log k,
\end{equation}
(see also \cite{cf:RO} where the technique is extended to obtain estimates on the mixing time for arbitrary graphs in terms of the mixing time of a single particle).

\medskip

In another direction: in \cite{cf:CLR}, it is shown that the spectral gap for the simple-exclusion on any graph is equal to that of the underlying simple random walk 
(e.g.\ in our case $\gl_1=2(1-\cos(2\pi/N))$).

Finally concerning the case of dimension $1$: in \cite{cf:Wilson}, the mixing time of the exclusion process on the segment is proved to be larger than 
$(2\pi^2)^{-1}N^2 \log k$ and smaller than  $(\pi^2)^{-1}N^2 \log k$,
with the conjecture that the lower bound is sharp. This conjecture was proved in \cite{cf:Lac}.

\subsection{The main result}

The first result of this paper is a sharp asymptotic for the mixing time of the exclusion process on the circle $\bbZ_N$.
For a fixed $\gep\in(0,1)$, when $N$ and $k$ goes to infinity we are able to identify the asymptotic behavior of $\Tm(\gep)$.
We obtain that when $k\le N/2$ (which by symmetry is not a restriction)

$$\Tm(\gep)=\frac{N^2}{8\pi^2}(\log k)(1+o(1)).$$

Note that here the dependence in $\gep$ is not present in the asymptotic equivalent.
This means that on a time window which is $o(N^2\log k)$ the distance to equilibrium drops abruptly from $1$ to $0$.
This sudden collapse to equilibrium for a Markov chain was first observed by Diaconis and Shahshahani \cite{cf:DS} in the case of the (mean-field) transposition shuffle (see also \cite{cf:Aldous} for the random walk on the hypercube). 
The term cutoff itself was coined in \cite{cf:AD} and the phenomenon has since been proved to hold for a diversity of
Markov chains (see e.g.\ \cite{cf:LS, cf:LS3} 
for some recent celebrated papers proving cutoffs). It is believed that cutoff holds with some generality for reversible Markov chains
as soon as the mixing time is much larger than the inverse of the spectral gap,
but this remains a very challenging conjecture (see \cite[Chapter 18]{cf:LPW} for an introduction to cutoff and some 
counterexamples and \cite{cf:bdcutoff, cf:Chen, cf:Basu} for recent progress on that conjecture).

\medskip

A natural question is then of course: ``On what time scale does $d(t)$ decrease from, say, $999/1000$ to $1/1000$ ?'' .
This is what is called the cutoff window. We are able to show it is equal to $N^2$.
Let us mention that, this is result is, to our knowledge, 
the first sharp derivation of a cutoff window for a lattice interacting particle system.

\begin{theorem}\label{mainres}
For any sequence $k(N)$ satisfying $k(N)\le N/2$ and tending to infinity.
We have for every $\gep\in (0,1)$
\begin{equation}\label{cutoff}
\lim_{N\to \infty} \frac{8\pi^2\Tm^{N,k}(\gep)}{N^2\log k}=1.
\end{equation}
More precisely we have

\begin{equation}\label{damix}\begin{split}
\lim_{s\to \infty} &\limsup_{N\to \infty} d^{N,k} \left( (8\pi^2)^{-1}N^2\log k+sN^2 \right)=0,\\
\lim_{s\to -\infty} &\liminf_{N\to \infty}  d^{N,k} \left( (8\pi^2)^{-1}N^2\log k+sN^2 \right)=1
\end{split}\end{equation}
and the window is optimal in the sense that for any $u\in \bbR$ 
\begin{equation}\label{damix2}\begin{split}
&\limsup_{N\to \infty}  d^{N,k} \left( (8\pi^2)^{-1}N^2\log k+uN^2 \right)<1,\\
&\liminf_{N\to \infty}  d^{N,k} \left( (8\pi^2)^{-1}N^2\log k+uN^2 \right) >0.
\end{split}\end{equation}
\end{theorem}

\begin{rem}
The result above can be reformulated in the following manner: 
\begin{equation}\begin{cases}
 &\Tm= \Tm(1/4)= (8\pi^2)^{-1}N^2\log k+O(N^2),\\
 &\forall \gep\in (0,1), \quad  \limsup_{N\to \infty} \frac{|\Tm(\gep)- \Tm|}{N^2}<\infty,\\
 &\lim_{\gep \to 0} \liminf_{N \to \infty} \frac{\Tm(\gep)- \Tm}{N^2}= \lim_{\gep \to 0} \liminf_{N \to \infty} \frac{\Tm-\Tm(1-\gep)}{N^2}=\infty.
\end{cases}\end{equation}
The second line states that the cutoff window is at most $N^2$ while the third implies not only that this is sharp, but also that one the time scale $N^2$,
the ``cutoff profile'' has infinite support in both directions.
\end{rem}

\begin{rem}
Our result does not cover the case of a bounded number of particles. In this case there is no cutoff and the mixing time is of order $N^2$ for every $\gep$
with a pre-factor which depends on $\gep$ (a behavior very similar to the random-walk: case $k=1$).
\end{rem}

We also show that on the other-hand that starting from a typical configuration, the relaxation to equilibrium is not abrupt and occurs on the time-scale $N^2$.

\begin{theorem}\label{secondres}
For any sequence $k(N)$ satisfying $k(N)\le N/2$ and tending to infinity, we have for all $u>0$
\begin{equation}
 0< \liminf_{N\to \infty}  {\bf d}^{N,k}(N^2u) \le \limsup_{N\to \infty} {\bf d}(N^2u)<1
\end{equation}
and 
\begin{equation}
\lim_{s\to \infty} \limsup_{N\to \infty} {\bf d}^{N,k}(N^2s)=0.
\end{equation}
\end{theorem}

\begin{rem}
Note that we will not prove that 
\begin{equation}\label{liminf}
\lim_{s\to 0}\liminf_{N\to \infty}   {\bf d}(N^2s)=1,
\end{equation}
 which would complete the picture by showing that the system does not mix at all before the $N^2$ time-scale.
However we point out to the interested reader that such a result can be obtained combining ingredients of Section \ref{lowerbounds}  together with \cite[Lemma 3.1]{cf:Lac2}
which asserts that the fluctuations of  $a_1(\eta)$ defined in \eqref{defa1} are asymptotically Gaussian. The convergence of ${\bf d}(N^2\cdot)$ when $N$ tends to infinity remains an open question.
 \end{rem}

\subsection{The cutoff for the exclusion on the segment}

Let us, in this section, briefly sketch the proof or at least recall the ingredients used in \cite{cf:Lac} 
to derive the cutoff for the exclusion on the segment $ \llbracket 1,N  \rrbracket$.
The state-space of particle configuration on the segment comes with a natural order 
\begin{equation}\label{order}
 \eta\le \eta'  \quad \Leftrightarrow \quad \forall x\in  \llbracket 1,N  \rrbracket \quad  \sum_{y=1}^x \eta_y \le  \sum_{y=1}^x \eta'_y.
\end{equation}

We set $\xi_x(\eta):= \sum_{y=1}^x \eta_y.$
It turns out that not only this order is preserved by the dynamics in a certain sense (see \cite{cf:Wilson} but also \cite{cf:Rost}), 
and but it has additional properties: if $\wedge$ denotes the maximal configuration for \eqref{order} then $P_t(\wedge,\cdot)$ is an increasing function (for any $t$); 
we also have positive correlation of increasing events (a.k.a \ the FKG inequality after \cite{cf:FKG}).

\medskip

These monotonicity properties are first used to show that after time of $(1+\delta)(2\pi)^{-1}N^2\log k$, starting the dynamics from  $\wedge$, 
we can couple a finite dimensional
(\textit{i.e.} whose dimension remains bounded when $N$ grows)
 projection of 
$(\xi_x(\eta_t))_{x\in \llbracket 1,N  \rrbracket}$ together with the corresponding equilibrium distribution.

\medskip

The monotonicity is then used again to check that the Peres-Winckler censorship inequality \cite[Theorem 1.1]{cf:PW} is valid in our context. The latter result establishes that, if 
one starts from the maximal configuration for the order described in \eqref{order}, ignoring some of the updates in the dynamics only makes the mixing slower
(as shown in \cite{cf:Hol}, this can fail to be true if there is no monotonicity). 
We use this statement to show that the system mixes in a 
time much smaller than  $(2\pi)^{-1} N^2\log k$ once a finite projection is close to equilibrium.
This method establishes a sharp upper bound (at first order) for the mixing time starting from $\wedge$ and some additional work is necessary to show that this is indeed the worse initial condition
(we refer to the introduction \cite{cf:Lac} for more details).

\subsection{Differences between the segment and the circle}

In the view of the previous section, the proof in \cite{cf:Lac} for the mixing time for the exclusion on the segment heavily relies on monotonicity arguments
in every step of the process.
The drawback of this approach is that it is not very robust, and cannot be used for either higher dimension graphs 
(for instance $\{1,\dots,N\}^d$ with either free or periodic boundary condition). It even breaks down completely if one allows jump between site $1$ and $N$.

\medskip

With this in mind, our idea when studying the exclusion on the circle is also to develop an approach to the problem which is more flexible, and 
could provide a step towards the rigorous identification of the cutoff threshold in higher dimensions 
(see Section \ref{higherdef} for conjectures and rigorous lower-bounds).
This goal is only partially achieved as, even if we do not require monotonicity, 
a part of our proof relies on the interface representation of the process (see Section \ref{fluctuat})
which is a purely one-dimensional feature. However let us mention that a $d$ dimensional generalization of Proposition \ref{smallfluctu} 
can be shown to remain valid for $d\ge 2$. A missing ingredient in higher dimension is thus a coupling which allows to couple particle configuration with typical fluctuation with equilibrium in a time of order $N^2$.

\medskip

Another positive point is that by relying much less on monotonicity, we are able to prove statements about the mixing time starting from an arbitrary position (cf.
Theorem \ref{secondres}) instead of focusing only on the extremal ones.

\medskip

Finally note that the method developed in this paper gives more precise results than the one in \cite{cf:Lac} as we identify exactly the width of 
the cutoff window (and it also extends to the segment).
However, we could not extract from it the asymptotic mixing time for the adjacent transposition shuffle, which seem to require novel ideas.

\begin{rem} In \cite{cf:Lac2},
by combining the technique of the present paper with some aditionnal new ideas, the author improved Theorem \ref{mainres} by describing the full cutoff profile, 
that is, identified the limit of  $d^{N,k} \left( (8\pi^2)^{-1}N^2\log k+uN^2 \right)$. 
Proposition \ref{smallfluctu} as well as the multiscale analysis used 
in Section \ref{multis} play a crucial role in the proof.
\end{rem}

\subsection{Organization of the paper}

In Section \ref{lowerbounds} we prove the part of the results which corresponds to lower-bounds for the distance to equilibrium, that is to say, the first lines of 
\eqref{damix} and \eqref{damix2}.
The proof of this statement is very similar to the one proposed by Wilson in \cite{cf:Wilson}, 
the only significant difference is that we have to work directly with the particle configuration instead of the height-function. 
Doing things in this manner underlines that the proof in fact does not rely much on the dimension   (see Section \ref{higherdef}).
While the proof does not present much novelty, 
we prefer to mention it in full as it is relatively short and it improves the best existing bound in the literature (see \cite{cf:Mor}).

\medskip

The main novelty in the paper is the strategy to prove upper-bound results (second lines of \eqref{damix} and \eqref{damix2}).
In Section \ref{pubdecompo} we explain how the proof is decomposed. In Section \ref{arratia}, we use a comparison inequality of Liggett \cite{cf:Lig77}
to control the (random) fluctuations of the  local density of particle after time $\frac{N^2}{8\pi^2}\log k$. Finally 
we conclude by showing that configuration which have reasonable fluctuations couples with equilibrium within time $O(N^2)$, 
using interface representation for the particle system,
and a coupling based on the graphical construction. The construction is detailed in Section \ref{fluctuat}, 
 and the proof is performed using
a multi-scale analysis in Section \ref{multis}.  

\section{Lower bound on the mixing time}\label{lowerbounds}

\subsection{The statement}

The aim of this Section is to prove the some lower bounds on the distance to equilibrium.
Following the method of \cite[Theorem 4]{cf:Wilson}, we achieve such a bound by controlling the first two moments of the first Fourier coefficient of $\eta_t$.

\begin{proposition}\label{dastate}
For any sequence $k(N)$ satisfying $k(N)\le N/2$ and tending to infinity.  we have 
\begin{equation}\label{laslande}
\lim_{s\to -\infty} \lim_{N\to \infty}  d^{N,k} \left( (8\pi^2)^{-1}N^2\log k+sN^2 \right)=1,
\end{equation}
and for any $u\in \bbR$
\begin{equation}\label{wiltord}
\liminf_{N\to \infty} d^{N,k} \left( (8\pi^2)^{-1}N^2\log k+uN^2 \right) >0.
\end{equation}
Moreover we have for any $u>0$
\begin{equation}\label{meanstuf}
\liminf_{N\to \infty}  {\bf d}^{N,k}(N^2u)>0.
 \end{equation}
\end{proposition}

\subsection{Relaxation of the``first" Fourier coefficient}

The main idea is to look at ``the first" Fourier coefficient (a coefficient corresponding to the smallest eigenvalue 
of the discrete Laplacian on $\bbZ_N$), of $\eta_t$.
For $\eta\in \gO_{N,k}$, we define
\begin{equation}\label{defa1}
a_1(\eta):=\sum_{x\in \bbZ_N} \eta(x) \cos \big(\frac{2\pi x}{N}\big).
\end{equation}
It is  an eigenfunction of the generator $\mathcal L$, (the reason for this being that each particle performs a diffusion for which 
$\cos (\frac{x2\pi}{N})$ is an eigenfunction), associated to the eigenvalue $-\gl_1$ where 
\begin{equation}
\gl_1:=2\left(1-\cos(2\pi/N)\right).
\end{equation}

\begin{lemma}
The function $a_1$ is an eigenfunction of the generator $\mathcal L$ with eigenvalue $-\gl_1$, and as a consequence,
for any initial condition $\chi\in \gO$
\begin{equation}
M_t:=e^{-t\gl_1}a_1(\eta_t^\chi)
\end{equation}
is a martingale for the filtration $\cF$ defined by
$$\cF_t:=\sigma((\eta_s)_{s\le t} ).$$
In particular we have
\begin{equation}
\bbE\left[ a_1(\eta_t^\chi)\right]=e^{-t\gl_1}a_1(\chi).
\end{equation}
Furthermore one can find a constant such that for all $t\ge 0$
\begin{equation}\label{lavar}
\var \left[ a_1(\eta_t^\chi)\right]\le 2k.
\end{equation}

\end{lemma}

\begin{proof}
Using the notation 
\begin{equation*}
\nabla f(x)=f(x+1)-f(x) \quad \text{and} \quad \gD f=f(x+1)+f(x+1)-2f(x), 
\end{equation*}
and set $\overline \cos (x):=\cos \big(\frac{2\pi x}{N}\big)$.
We have  

\begin{multline}
\mathcal L a_1(\eta):= \sum_{x\in \bbZ_N} \left(a_1(\eta^x)-a_1(\eta)\right)=
-\sum_{x\in \bbZ_N} \nabla \eta(x)\nabla \overline \cos (x)
\\=\sum_{x\in \bbZ_N} \eta(x)\gD \overline \cos (x)= -\gl_1 a_1(\eta)
\end{multline}
where the second equality comes from reindexing the sum and the last one from the identity 
\begin{equation}
\gD \overline \cos (x) =-\gl_1 \overline \cos (x).
\end{equation}
From the Markov property, we have for every positive $t$,
\begin{equation}
\partial_s \bbE[M_{t+s} \ | \ \cF_t ] |_{s=0}= \gl_1 M_{t}+ e^{t\gl_1}(\cL a_1)(\eta^\chi_t)=0.
\end{equation}
which implies that it is a martingale. In particular we have 
\begin{equation}\label{esp}
\bbE \left[a_1(\eta^{\chi}_t)\right]= e^{-\gl_1 t}a_1(\chi).
\end{equation}

Now let us try to estimate the variance of $M_t$: for the process with $k$ particle, the maximal transition rate is $2k$
(each of the $k$ particles can jump in an most $2$ directions independently with rate one).
If a transition occurs at time $s$, the value of  $M_s$ varies at most by an amount

$$e^{\gl_1 s}\max_{x\in \bbZ_N} |\bcos (x)-\bcos (x+1)|\le e^{\gl_1 s}\frac{2\pi}{N}.$$

With this in mind we can obtain a bound on the bracket of $M$ (that is: the predictable process such that $M^2_t-\langle M\rangle_t$ is a martingale) 

\begin{equation}
\langle M\rangle_t \le 2k\int_0^t e^{2\gl_1 s}\left(\frac{2\pi}{N}\right)^2\dd s
\end{equation}
Then using the fact that $\var(M_t)= \bbE\left[ \langle M\rangle_t \right]$, we 
have, for $N$ sufficiently large, for any $\chi\in \gO_{N,k}$ and any $t\ge 0$
\begin{equation}\label{varance}
\var \left[ a_1(\eta_t^\chi)\right]=e^{-2\gl_1 t}   \bbE\left[ \langle M\rangle_t \right]
\le 2k  \int_0^t  e^{2\gl_1 (s-t)} \left(\frac{2\pi}{N}\right)^2\dd s\le \frac{4\pi^2k}{N^2\gl_1}\le 2k
\end{equation}
where the last inequality comes from the fact that $\gl_1\sim 4\pi^2 N^{-2}$.
\end{proof}

At equilibrium (\textit{i.e.} under the distribution $\mu$) $a_1(\eta)$ has mean zero and typical fluctuations of order $\sqrt{k}$. The equilibrium variance can either be computed directly or one can use \eqref{lavar} for $t\to \infty$ to obtain

$$\var_{\mu}\left(a_1(\eta) \right)\le 2k.$$
From \eqref{lavar}, if $\bbE\left[a_1(\eta_t^\chi)\right] $ is much larger than $\sqrt{k}$ then $a_1(\eta_t^\chi)$ is much larger than $\sqrt{k}$ with large probability which implies
$\| P^{\chi}_t-\mu\|$ has to be large. We need to use this reasoning for a $\chi$ which maximizes $a_1$.

\subsection{Proof of Proposition \ref{dastate}}

Using \cite[Proposition 7.8]{cf:LPW} (obtained from the Cauchy Schwartz inequality)
and the estimates \eqref{esp}-\eqref{varance}, we have

\begin{equation}\begin{split}\label{kit}
\|P^{\chi}_t-\mu\|_{TV}&\ge \frac{\left(\bbE\left[a_1(\eta_t^{\chi})\right]\right)^{2}}{\left(\bbE\left[a_1(\eta_t^{\chi})\right]\right)^{2}+
2\left[\var \left( a_1(\eta_t^{\chi})\right)+\var_{\mu} \left( a_1(\eta) \right) \right]}\\
&\ge\frac{1}{1+ 8k \exp(2\gl_1 t) a_1(\chi)^{-2}}.
\end{split}\end{equation}

Consider $\chi=\chi_0$ being the configuration which minimizes $a_1$.
\begin{equation}
\chi_0(x):=\begin{cases} \ind_{\{x\in \{-p,\dots, p\}\}} \text{ if } k=2p+1,\\
\ind_{\{x\in \{-p+1,\dots, p\}\}} \text{ if } k=2p,
\end{cases}
\end{equation}
It is rather straight-forward to check that for any $N\ge 2$, $k\le N/2$, we have 
$a_1(\chi_0)\ge k/2$.
Thus using the above inequality for $t=t_N:= uN^2+ \frac{N^2}{8\pi^2}\log k$  ($u\in \bbR)$ the reader can check that

\begin{equation}
\liminf_{N\to \infty} \|P^{\chi_0}_{t_N}-\mu\|_{TV}\ge  \lim_{N\to \infty} \frac{1}{1+32 k^{-1} \exp(-2\gl_1 t_N)}= \frac{1}{1+32e^{- 8\pi^2 u}},
\end{equation}
which implies both \eqref{laslande} and \eqref{wiltord}.
To prove \eqref{meanstuf} we need to use \eqref{kit} for the  set 
$$A^N_{\delta}:=\{ \chi\in \gO_{N,k} \ | \ a_1(\chi)\ge \delta \sqrt{k}\}.$$ 
We have 
\begin{equation}
{\bf d}(sN^2)\le \left[ 1-\mu( A_{\delta}) \right] +\frac{\mu( A_{\delta})}{1+ 8 \delta^{-2}\exp(2\gl_1N^2 s) }.
\end{equation}
To conclude the proof, it is sufficient to prove that $\liminf_{N\to \infty} \mu_N(A_{\delta})>0$ for some small $\delta>0$. 
This can be done e.g. by showing that $\mu\left[ (a_1(\eta))^4\right]\le Ck^2$ (which is left as an exercise to the reader).

\qed

\subsection{The exclusion in higher dimensions} \label{higherdef}

Let us shortly present in this section a generalization of Proposition \ref{dastate} for the exclusion process in higher dimension
$d\ge 2$.

\medskip

For $N\in \bbN$, and $k\le N^d/2$ we define the state-space of particle configurations as
\begin{equation}
\gO^d_{N,k}:=\left\{ \eta\in \{0,1\}^{\bbZ_N^d} \ | \sum_{x\in \bbZ_N^d} \eta(x)=k \right\}.
\end{equation}
Given $x\sim y$ a pair of neighbor on the torus $\bbZ_N^d$, we set
\begin{equation}
\eta^{x,y}:=\begin{cases} \eta^{x,y}(x)=\eta(y),\\
 \eta^{x,y}(y)=\eta(x)\\
  \eta^{x,y}(z)=\eta(z), \quad \text{ for } z\notin \{x,y\}.
\end{cases}
\end{equation}
and define the generator by 
\begin{equation}
\mathcal L f(\eta):= \sumtwo{x, y\in \bbZ_N}{x\sim y} f(\eta^{x,y})-f(\eta).
\end{equation}

We set $d^{N,k,d}$ to be the distance to equilibrium of the chain with generator $\cL$ at time $t$ (see \eqref{dat}).
The we can adapt the proof of Proposition \ref{dastate} and show that

\begin{proposition}
For any sequence $k(N)$ satisfying $k(N)\le N^d/2$ and tending to infinity.  we have 
\begin{equation}
\lim_{u\to \infty} \lim_{N\to \infty}  d^{N,k,d} \left( (8\pi^2)^{-1}N^2\log k-uN^2 \right)=1,
\end{equation}
and for any $u\in \bbR$
\begin{equation}
\liminf_{N\to \infty} d^{N,k,d} \left( (8\pi^2)^{-1}N^2\log k+uN^2 \right) >0.
\end{equation}
\end{proposition}

\begin{rem}
Note that the result remains valid if the torus is replaced by the grid (\textit{i.e.} if we drop the periodic boundary condition) in which case 
$(8\pi^2)^{-1}$ has to be replaced by $(2\pi^2)^{-1}$.
In view of this result, and of the content of the next section, it is natural to conjecture that
$(8\pi^2)^{-1}N^2\log k$ is the mixing time of the exclusion process on the torus.
 \end{rem}
\begin{proof}
The proof is almost exactly the same.
The eigenfunction which one has to consider is 
\begin{equation}
a_1(\eta):=\sum_{x\in \bbZ^d_N} \eta(x) \cos \left(\frac{x_1\pi}{N}\right).
\end{equation}
where $x_1\in \bbZ_N$ is the first coordinate of $\bbZ_N$.
It is not difficult to check that if $\chi_0$ is a maximizer of $\bbZ_N^d$ (there might be many of them)
$a_1(\chi^0)$ is larger than $k/2$.
\end{proof}

 \section{Upper bound on the mixing time} \label{pubdecompo}

 \subsection{Decomposition of the proof}

To complete the proof of the main result, we have to prove 

\begin{proposition}\label{dastate2}
For any sequence $k(N)$ satisfying $k(N)\le N/2$ and tending to infinity, and for any $u\in \bbR$,  we have
\begin{itemize} 
\item[(i)]
\begin{equation}\label{mikou}
\lim_{s \to \infty} \limsup_{N\to \infty}  d^{N,k} \left( (8\pi^2)^{-1}N^2\log k+sN^2 \right)=0,
\end{equation}
\item[(ii)]
\begin{equation}\label{benarbia}
\limsup_{N\to \infty} d^{N,k} \left( (8\pi^2)^{-1}N^2\log k+uN^2 \right) <1.
\end{equation}
\item[(iii)]
\begin{equation}\label{mikou2}
\lim_{s\to \infty} \limsup_{N\to \infty}  {\bf d}(sN^2)=0,
\end{equation}
\item[(iv)]
\begin{equation}\label{benarbia2}
\limsup_{N\to \infty}  {\bf d}(uN^2)<1.
\end{equation}
\end{itemize}
\end{proposition}

The proof of this statement is much more involved than that of Proposition \ref{dastate} and 
relies on an explicit coupling of $P^{\chi}_t$ and the equilibrium measure $\mu$ for an arbitrary $\chi\in \gO$, which requires two 
steps.

\medskip

In a first step we want to show that after time $t_0=(8\pi^2)^{-1}N^2\log k$, or even shortly before that time, the density of particle is close to $k/N$
everywhere on the torus and that the deviation from it are not larger than equilibrium fluctuation (which are of order $\sqrt{k}$).
This part of the proof relies on comparison inequalities developed by Liggett \cite{cf:Lig77}, which allow to replace 
the exclusion process with $k$ independent random walks.

\medskip

In a second step, we construct a dynamical coupling of the process starting $\chi$ which has fluctuations of order $\sqrt{k}$, with one starting from 
equilibrium, using the height-function representation. We show that the two height functions couple within a time $O(N^2)$ which is what we need to conclude.
 The construction of the coupling and heuristic explanations are given in Section \ref{fluctuat}, while the proof is performed in Section \ref{multis}.

\subsection{Control of the fluctuation of the particle density}

To present the main proposition of the first step we need to introduce some notation
Given $x\ne y$ in  $\bbZ_N$, we define the interval $[x,y]$ to be the smallest (for the inclusion) subset $I$ of $\bbZ_N$ which contains $x$ 
and which satisfies
\begin{equation}\label{theinterval}
\forall z\in I\setminus\{y\},\ z+1\in I.
\end{equation}
Let $f$ be a function defined on $\bbZ_N$ we use the notation
\begin{equation}
\sum_{z=x}^y f(z):=\sum_{z\in [x,y]} f(z).
\end{equation}
We define the \textsl{length} of the interval (which we write $\#[x,y]$) to be the number of points in it (e.g.\ it is equal to $y-x+1$ if $1\le x\le y\le N$). 
We will prove the following result:
given $A\ge 0$ we set 
\begin{equation}
t_A=\frac{N^2}{8\pi^2} \log k- A N^2.
\end{equation}

\begin{proposition}\label{smallfluctu}
There exists a constant $c$ such that, 
for all $A\in \bbR$, for all $N$ sufficiently large (depending on $A$)
for all initial condition $\chi\in \gO_{N,k}$,
\begin{equation}\label{fluqueton}
\bbP\left[\exists x,y \in \bbZ_N, \
\left| \sum_{z=x+1}^y \left(\eta^\chi_{t_A}(z)-\frac{k}{N}\right)\right|\ge \left(s+8e^{4\pi^2A}\right) \sqrt{k} 
\right] \le 2\exp\left( -c s^2\right).
\end{equation}
\end{proposition}

\subsection{Coupling with small fluctuations}

In the second step of our proof, we show that starting from a 
configuration with small fluctuations we can relax to equilibrium within time $O(N^2)$.
Set 
\begin{equation}\label{etagro}
\mathcal G_s := \left\{ \eta\in \gO  \ | \ \forall x, y \in \bbZ_N   \left|\sum_{z=x+1}^y \left(\eta^\chi_t(z)-\frac{k}{N}\right)\right|\le s \sqrt{k}\right\}
\end{equation}
The following proposition establishes this diffusive relaxation to equilibrium in two ways:
first it shows that one gets $\gep$ close to equilibrium within a time $C(s,\gep)$, but also that on the scale $N^2$
the distance becomes immediately bounded away from one for positive times.

\begin{proposition}\label{csdf}
For any $s\ge1$, given $\gep>0$ there exists a constant $C(s,\gep)$ such that 

\begin{equation}
\forall \chi \in \mathcal G_s,\ \| P^{\chi}_{C(s,\gep) N^2}-\mu \|_{TV}\le  \gep.
\end{equation}
For any $s,u>0$, there exists $c(s,u)>0$ such that 
\begin{equation}\label{gramicho}
\forall \chi \in \mathcal G_s,\ \| P^{\chi}_{uN^2}-\mu \|_{TV}\le 1-c(s,u).
\end{equation}
\end{proposition}
Now we show that Propositions \ref{smallfluctu} and \ref{csdf} are sufficient to prove \eqref{dastate2}.

\begin{proof}[Proof of Proposition \ref{dastate2}]
 We use the semi-group property at time $t_A$.
We have for any $\chi\in \gO$

\begin{equation}
P_{t_A+CN^2}^\chi(\cdot)=\sum_{\chi'\in \gO} P_{t_A}^{\chi}(\chi')P^{\chi'}_{CN^2}\left(\cdot \right).
\end{equation}

Hence, using the triangle inequality, we have  for any event $\mathcal G$

\begin{multline}\label{croot}
\| P_{t_A+CN^2}^\chi-\mu\|\le \sum_{\chi'\in \gO}P_{t_A}^{\chi}(\chi') \| P^{\chi'}_{CN^2}-\mu \|
\le
 P_{t_A}^{\chi}(\mathcal G^c)+P_{t_A}^{\chi}(\mathcal G)\max_{\chi'\in \mathcal G}\| P^{\chi'}_{CN^2}-\mu \|
\end{multline}
We can now start the proof of \eqref{mikou}. According to Proposition \ref{smallfluctu}, if $s$ is sufficiently large,
 we have
 \begin{equation}
 P_{t_0}^{\chi}((\mathcal G_s)^c)\le \gep/2.
 \end{equation}
Fixing such an $s$ (which we denote by $s(\gep)$),
 according to Proposition \ref{csdf} we can find a constant $C(\gep)$ such that 
\begin{equation}
\max_{\chi'\in \mathcal G_s(\gep)}\| P^{\chi'}_{C(\gep)N^2}-\mu \|\le \gep/2,
\end{equation}
which is enough to conclude the proof,  using \eqref{croot} with $A=0$, $\mathcal G=\mathcal G_{s(\gep)}$, and $C=C(\gep)$.\\
To prove \eqref{mikou2}, we note that for any $\cG$
\begin{equation}\label{croot2}
{\bf d}^{N,k}(uN^2)\le \mu(\mathcal G^c)+ \mu(\mathcal G) \max_{\chi'\in \mathcal G_s}\| P^{\chi'}_{uN^2}-\mu \|
\end{equation}
and we can conclude similarly using Proposition \ref{smallfluctu} to find $s(\gep)$ such that that $\mu(\cG^c_{s(\gep)})<\gep/2$ for all $N$, and then $u$ large enough.
\medskip

We now prove \eqref{benarbia}. For a fixed $u<-1$, for $A=1-u$, we can find $s(u)$ sufficiently large such that
 \begin{equation}
 P_{t_A}^{\chi}(\mathcal G_{s(u)})\ge \frac 1 2.
 \end{equation}
Using \eqref{gramicho} we obtain that 
\begin{equation} 
\max_{\chi \in \mathcal G_{s(u)} }\quad   \| P^{\chi}_{N^2}-\mu \|_{TV}\le 1-c(s(u),1).
\end{equation}
We can then conclude by using \eqref{croot} for $A=s-1$ and $ \mathcal G=\mathcal G_{s(u)}$, that for large $N$
\begin{equation}
d^{N,k} \left( (8\pi^2)^{-1}N^2\log k+uN^2 \right) <1-\frac{c(s(u),1)}{2}.
\end{equation}
Concerning \eqref{benarbia2}, we choose $s_0$ such that $\mu(\cG^c_{s_0})\ge 1/2$, and  use \eqref{gramicho} and \eqref{croot2} to show that 
\begin{equation}
{\bf d}^{N,k} (uN^2)\le 1-\frac{c(s_0,u)}{2}.
\end{equation}

\end{proof}

\section{Proof of Proposition \ref{smallfluctu}} \label{arratia}

Let us first slightly modify the statement

\begin{proposition}\label{secondfluctu}
There exists a constant $c$ such that
for all $N$ sufficiently large, for all $t\ge 3N^2$
for all $\chi\in \gO_{N,k}$, for all $s\ge 0$
\begin{equation}\label{fluqueton2}
\bbP\left[\exists x,y \in \bbZ_N, \
\left| \sum_{z=x+1}^y (\eta^\chi_t(z)-\bbE[\eta^\chi_t])\right|\ge s \sqrt{k} 
\right] \le 2\exp\left( -c s^2\right).
\end{equation}
\end{proposition}

Proposition \ref{smallfluctu} can be deduced from Proposition \ref{secondfluctu} using the following Lemma which relies on 
heat-kernel estimates. The proof is based on the diagonalization of the Laplace operator on the discrete circle. The method is classic (see e.g \cite{cf:Fort}),
but as we could not find a reference matching exactly our setup and result, we include the proof in the appendix for the sake of completeness.

\begin{lemma}\label{bound}
The following statement hold
\begin{itemize}
\item [(i)] For $N$ large enough, for all $\chi$, all $x\in \bbZ_N$ and $t\ge N^2$
\begin{equation}
\left| \bbE[\eta^\chi_t(x)]-\frac{k}{N} \right| \le 4kN^{-1} e^{-\gl_1 t}.
\end{equation}
In particular 
\begin{equation}\label{labiound}
 \bbE[\eta^\chi_t(x)]\le \frac{2k}{N}.
\end{equation}

\item [(ii)]  If $X_t$ is a nearest neighbor random walk  starting from $x_0\in \bbZ_N$ one has for all $t\ge N^2$, all 
$x,y\in \bbZ_N$
\begin{equation}\label{labiound2}
 \bbP\big[X_t\in [x,y] \big]\le \frac{2\#[x,y]}{N}
\end{equation}
\end{itemize}

\end{lemma}

\begin{proof}[Proof of Proposition \ref{smallfluctu}]
We have for all $N$ sufficiently large and $x,y \in \bbZ_N$
\begin{equation}
\sum_{z=x}^y\left| \bbE[\eta^\chi_{t_A}(z)]-\frac{k}{N} \right|\le 4\#[x,y] kN^{-1}e^{\gl_1 t_A}\le 8\sqrt{k} e^{4\pi^2A}.
\end{equation}
Hence we have 
\begin{equation}
\left| \sum_{z=x}^y \eta^\chi_{t_a}(z)-\frac{k}{N} \right|\le \sum_{z=x}^y\left| 
\eta^\chi_{t_A}(z)-\sum_{z=x}^y \bbE[\eta^\chi_{t_A}(z)] \right|+ 8\sqrt{k} e^{4\pi^2A}.
\end{equation}
and thus Proposition \ref{smallfluctu} follows from Proposition \ref{secondfluctu}.
\end{proof}

\begin{rem}
Note that by taking $t= \infty$ with $A=0$ in \eqref{fluqueton2}, we also have a result concerning 
density fluctuation for the equilibrium measure $\mu$ which we will use during our proof.
\begin{equation}\label{fluquetec}
\mu\left( \exists x,y \in \bbZ_N, \
\left| \sum_{z=x+1}^y (\eta(z)-\frac{k}{N})\right|\ge s \sqrt{k} 
\right) \le 2\exp\left( -c s^2\right).
\end{equation}
\end{rem}

The idea of the proof is to control the Laplace transform of the number of particle in each interval (and then we roughly have to sum over all intervals to conclude).
To control this Laplace transform, with use a comparison inequality due to Liggett \cite{cf:Lig77}, which allows us to compare the simple exclusion with a particle system without exclusion, that is: $k$ independent random walks on the circle.
With this comparison at hand, the Laplace transform can be controlled simply by using \eqref{labiound2}.

\subsection{Estimate on the Laplace transform}

From now on the initial condition $\chi$ is fixed, and for convenience, does not always appear in the notation.
For $x\in\bbZ_N$ ($x\in \{1,\dots N-1\}$) and we set 
\begin{equation}\begin{split}\label{defst}
S_{x,y}(t)&:= \sum_{z=x+1}^y (\eta^\chi_t(z)-\bbE[\eta^\chi_t]),\\
S_{x}(t)&:=S_{0,x}(t),
\end{split}
\end{equation}
\begin{lemma}\label{laplacetrans}
For all $x\in \bbZ_N$ and $t\ge N^2$
\begin{equation}\label{gronek}
\bbE\left[e^{\alpha S_x(\eta_t)}\right] \le \exp\left(2\frac{kx}{N} \alpha^2\right).
\end{equation}
\end{lemma}
\begin{rem}
Of course the formula \eqref{gronek} remains valid for any interval of length $x$ by translation invariance
\end{rem}

\begin{proof}[Proof of Proposition \ref{secondfluctu}]

.
Note that we can a always consider in the proof that $s$ is sufficiently large (as the result is obvious for $s\le ((\log 2)/c)^{1/2}$).
By the triangle inequality, we have for all $x,y\in \bbZ$.

\begin{equation}
\left|S_{x,y}(t)\right| \le   \left|S_x(t)\right|+\left|S_y(t)\right|
\end{equation}
For convenience we decide to replace $s$ by $16s$ in \eqref{fluqueton2} (this only corresponds to changing the value of $c$ by a factor $256$).
Hence it is sufficient to prove that for any $t\ge N^2$ we have
\begin{equation}\label{fluctuation}
\bbP\left[\exists x \in \bbZ_N, \
\left| S_x(t)\right|\ge 8s \sqrt{k}
\right] \le 2\exp\left( -c s^2\right).
\end{equation}
Now let us show that we can replace $ x \in \bbZ_N$ by a smaller subset.
Let $q_0$ be such that 

\begin{equation}\label{defqz}
\frac{Ns}{\sqrt k} < 2^{q_0} \le \frac{2Ns}{\sqrt k}
\end{equation}
For $x\in\{0,\dots N-1\}$ (which we consider as an element of $\bbZ_N$, 
one can find $y\in 2^{q_0} \{0,\dots, \lceil (N-1) 2^{-q_0}\rceil \}$ (a multiple of $2^{q_0}$) such that 
$$y \le  x \le y^+:= \max(y+2^{q_0},N),$$
We have, from the definition of \eqref{defst}
\begin{equation}\label{grimberg}
\begin{split}
S_{x}(t)&\ge S_y(t)-\sum_{z=y+1}^x\bbP[\eta_t(z)=1],\\
S_{x}(t)&\le S_{y^+}(t)+\sum_{z=x+1}^{y^+}\bbP[\eta_t(z)=1]
\end{split}
\end{equation}
From \eqref{labiound} and $$(y^+-y)\le 2^q_0\le 2Ns/(\sqrt{k}),$$
the second term of both equations in \eqref{grimberg} is smaller than $4Ns/(\sqrt{k})$ and hence
\begin{equation}
|S_{x}(t)|\le \max\left( | S_y(t)|, |S_{y^+}(t) |\right)  + 4s\sqrt{k}.
\end{equation}
Thus, we can reduce \eqref{fluctuation} to proving
\begin{equation}\label{fluctuation2}
\bbP\left[\exists y \in  2^{q_0} \{0,\dots ,  \lfloor N 2^{-q_0} \rfloor \},\
 | S_y(t)| \ge 4s \sqrt{k}
\right] \le 2\exp\left( -c s^2\right).
\end{equation}
The next step relies on multi-scale analysis.
Let $p$ be such that $N\in (2^p,2^{p+1}]$.
Given $y\in 2^{q_0} \{0,\dots ,  \lfloor (N-1) 2^{-q_0} \rfloor \}$, we can decompose in base $2$ as follows
\begin{equation}
2^{q_0} y=:\sum_{q=0}^{p-q_0} \gep_q 2^{p-q},
\end{equation}
where $\gep_q\in \{0,1\}$.
We set  $y_{-1}:=0$  and $y_r:=\sum_{q=0}^{r} \gep_q 2^{p-q}$, for $r\in\{0,\dots,p-q_0\}$.\\
Using the triangle inequality again we have $|S_y(t)|\le \sum_{r=0}^{p-q_0}  |S_{y_{r-1},y_r}(t)|,$
and hence 
 \begin{equation}\label{carmniq}
   \left\{ |S_y(t)| \ge 4s \sqrt{k} \right\} \Rightarrow  \left\{  \exists  r\in\{0,\dots,p-q_0\}   |S_{y_{r-1},y_r}(t)| \ge \left(\frac{3}{4}\right)^{r} s\sqrt{k} \right\}.
\end{equation}
Thus, the proof of Proposition \ref{smallfluctu} can be reduced to show the following.
 
 \begin{lemma}\label{fixeddev}
Let us define

\begin{multline}
\mathcal H(s,t):= \Big\{\exists q\in\{q_0,\dots,p\},\  \exists y\in \{1,\dots,\lfloor N2^{-q}\rfloor\},\\  |S_{2^{q}(y-1), 2^{q}y}(t)| \ge \left(\frac{3}{4}\right)^{p-q} s\sqrt{N}\Big\}
\end{multline}
For every $t\ge N^2$ we have
\begin{equation}
\bbP[\mathcal H(s,t)]\le 2e^{-cs^2}
\end{equation}
\end{lemma}

Indeed from  \eqref{carmniq} and the reasoning taking place before, one has
$$   \Big\{  \exists x\in \bbZ_N,  |S_{x}(t)| \ge  8s\sqrt{N}\Big\} \subset \mathcal H(s,t).$$
\end{proof}

\begin{proof}[Proof of Lemma \ref{fixeddev}]
We have by union bound
\begin{multline}\label{croco}
\bbP[\cH(s,t)]\le  \sum_{q=q_0}^p \sum_{y=1}^{\lfloor N 2^{-q} \rfloor} \bbP\left[ |S_{2^{q}(y-1), 2^{q}y}(t)| \ge \left(\frac{3}{4}\right)^{p-q} s\sqrt{k}  \right]
\\ \le \sum_{q=0}^p  2^{p+1-q} \max_{y} \bbP\left[   |S_{2^{q}(y-1), 2^{q}y}(t)| \ge\left(\frac{3}{4}\right)^{p-q} s\sqrt{k}  \right].
\end{multline}
Thus we have to find a bound on 
$$\bbE\left[ |S_{2^{q}(y-1), 2^{q}y}(t)| \le \left(\frac{3}{4}\right)^{p-q} s\sqrt{k}  \right]$$
which is uniform in $y$ and is such that the sum in the second line of  \eqref{croco} is smaller than $2e^{-cs^2}$.
For what follows we can, without loss of generality consider only the case $y=1$, as all the estimates we use are invariant by translation on $\bbZ_N$.

\medskip

Using Lemma \ref{laplacetrans} and the Markov inequality, we have for any positive $\alpha\le \log 2$
\begin{equation}\label{chernov}
  \bbE\left[  |S_{2^{q}}(t)|\ge \left(\frac{3}{4}\right)^{p-q} s\sqrt{k} \right] 
 \le \exp\left( 2^{q+1}\alpha^{2}\frac{k}{N}-\alpha\sqrt{k}\left(\frac{3}{4}\right)^{p-q} s\right).
  \end{equation}
We can check that the right-hand side is minimized for
 $$\alpha=\alpha_{0}:=2^{-(q+2)}\left(\frac{3}{4}\right)^{p-q} s \frac{N}{\sqrt{k}},$$
Note that for all $q\ge q_0$, (recall \eqref{defqz}) one has
\begin{equation}
\alpha_{0}\le 2^{-(q_0+2)}\left(\frac{3}{4}\right)^{p-q_0} s \frac{N}{\sqrt{k}}\le  \frac{1}{4}\left(\frac{3}{4}\right)^{p-q_0}\le \log 2.
\end{equation}
This ascertains the validity of \eqref{chernov}, and hence 
\begin{equation}
  \bbE\left[  |S_{2^{q}}(t)|\ge \left(\frac{3}{4}\right)^{p-q} s\sqrt{k} \right]
\le e^{-s^{2}2^{q-3}N\left(\frac{3}{4}\right)^{2(p-q)}}.
\end{equation}
Using the fact that $N\ge 2^p$ we have
\begin{equation}
\bbP\left[ |S_{2^{q}}(t)| \ge \left(\frac{3}{4}\right)^{p-q} s\sqrt{k}\right]
\le e^{-\frac{s^{2}}{8}\left(\frac{9}{8}\right)^{(p-q)}}.
\end{equation}
Using this in \eqref{croco} allows us to conclude (choosing $c$ appropriately).

\end{proof}

\subsection{Proof of Lemma \ref{laplacetrans}}

We use a result of Liggett \cite{cf:Lig77} which provides a way to compare the simple exclusion with a simpler process composed of independent random walkers.
If $f$ is a symmetric function on $(\bbZ_N)^k$ and $\eta\in \gO_{N,k}$ we set
\begin{equation}\label{feta}
f(\eta):=f(y_1,y_2,\dots,y_k)
\end{equation}
where 
$$\{y_1,\dots,y_k\}:=\{x \ | \ \eta(x)=1\}.$$
The above relation defines $(y_1,\dots,y_k)$ modulo permutation which is sufficient for the definition \eqref{feta}.
We say that a function $f$ defined on $(\bbZ_N)^2$ is positive definite if and only if for all $\gb$ such that $\sum_{x\in \bbZ_N}\gb(x)=0$, we have
\begin{equation}\label{defpov}
\sum_{x,y\in \bbZ^2_N} \gb(x)\gb(y) f(x,y)\ge 0.
\end{equation}
We say that a function defined on $(\bbZ_N)^k$ is positive definite 
if all its two dimensional marginals are. Note in particular a function that can be written in the form
\begin{equation}\label{product}
f(x_1,\dots, x_k)= C\prod_{i=1}^k g(x_k)
\end{equation}
where $C$ is a positive constant and $g$ is a non-negative function on $\bbZ_N$, is definite positive.
\medskip
Given $\chi\in \gO_{N,k}$,
let ${\bf X}^{\chi}_t:=(X^{\chi}_1(t),\dots, X^{\chi}_k(t))$ denote a set of independent random walk on $\bbZ_N$, starting from initial condition 
${\bf x}^{\chi}:=(x^{\chi}_1,\dots,x^{\chi}_k)$ which satisfies
\begin{equation}\label{codin}
\{x^\chi_1,\dots,x^{\chi}_k\}:=\{x \ | \ \chi(x)=1\}.
\end{equation}
Of course \eqref{codin} defines $(x^{\chi}_1,\dots,x^{\chi}_k)$ only modulo permutation but this has no importance for what we are doing (e.g. we can fix $(x^0_1,\dots,x^0_k)$ to be 
minimal for the lexicographical order).

\begin{proposition}\label{ligett}
If $f$ is a symmetric definite positive then we have for all $t\ge 0$
\begin{equation}
 \bbE\left[f(\eta^\chi_t)\right]\le \bbE[f({\bf X}^{\chi}_t)]
\end{equation}
\end{proposition}

\begin{proof}
The proof in the case $k=2$ is detailed in \cite[Proof of Lemma 2.7]{cf:Lig74} perfectly adapts to the case of general $k$.
We include it here for the sake of completeness. We let $\cL'$ denote the generator of $k$ independent random walks and 
$Q_t$ the associated semigroup. We define the action of $P_t$ and $Q_t$ on functions as follows
\begin{equation}\begin{split}
P_tf(\chi)&:= \sum_{\eta \in\gO_{N,k}} P_t(\chi,\eta)f(\eta),\\
Q_tf({\bf x})&:= \sum_{{\bf y} \in (\bbZ_N)^k} Q_t({\bf x},{\bf y})f({\bf y}).
\end{split}\end{equation}
As $Q_t$ is invariant by permutations of the labels, $Q_tf$ is also is a symmetric function and we can also consider $Q_tf(\eta)$ for $\eta\in \gO$ (cf. \eqref{feta})

\medskip
Using the standard property of Markov semi-group $\partial_t P_t:=\cL P_t=P_t \cL$, 
we have 
\begin{equation}
 \left[Q_t(f)-P_t(f)\right](\chi)= \int^t_0 \partial_s \left[ P_{t-s} Q_s(f)(\chi)\right] \dd s\\=
 \int_{0}^t P_{t-s} \left( \cL'-\cL \right) Q_{s} f(\chi)\dd s.
\end{equation}
and our conclusion follows if we can prove that $\left( \cL'-\cL \right)P_{t-s} f(\chi)\ge 0$ for all $\chi$ (as composition by $P_{s-t}$ preserves positivity).
First note that $g:=Q_{s}(f)$ is definite positive. Indeed if $\sum_{x\in \bbZ_N}\gb(x)=0$ we have  
\begin{multline}\label{ffesd}
 \sum_{x_1,x_2\in \bbZ_N} \gb(x_1)\gb(x_2)Q_sf(x_1,x_2,\dots,x_k)\\
 =\sum_{y_3,\dots,y_N\in \bbZ_N} \left(\prod_{j\ge 3} p_t(x_j,y_j) \right)\left( \sum_{y_1,y_2\in \bbZ_N}  \gb'(y_1)\gb'(y_2)
 f(y_1,y_2,\dots,y_k)\right).
\end{multline}
where $p_t$ is the discrete heat kernel on the circle.
and $$\gb'(y):=\sum_{x\in \bbZ_N} p_t(x,y) \gb(x)$$ 
satisfies $\sum_{y\in \bbZ_N} \gb'(y)=0$. Thus the l.h.s of \eqref{ffesd} is non-negative as a sum of non-negative terms.
Then, notice that the generator $\cL'$ includes all the transition of $\cL$ but also allows particle to jump on a neighboring site even if it is occupied.
Hence we have, for all ${\bf x}$ with distinct coordinates 
\begin{equation}\label{dsasdaa}
 (\cL'-\cL)g({\bf x})= \sum_{\{i<j \ | \ x_i\sim x_j\}} g(\cdot,x_i,\cdot,x_i,\cdot)+g(\cdot,x_j,\cdot,x_j,\cdot)-2g(\cdot,x_i,\cdot,x_j,\cdot).
\end{equation}
where in the right-hand side, only the $i$-th and $j$-th coordinate appears in the argument of $g$.
Note that each term in the r.h.s.\  of \eqref{dsasdaa} is of the form \eqref{defpov} with $\gb(x):= \ind_{\{x=x_i\}}-\ind_{\{x=x_j\}}$ 
(recall that $g$ is symmetric) and thus is positive.
\end{proof}

We want to apply Proposition \ref{ligett} to the function 
\begin{equation}
f(x_1,\dots,x_k):=e^{\alpha \sum_{i=1}^k \left(\ind_{\{x_i\in [1,y]\}}-\bbP\left[X_i(t)\in[1,y]  \right] \right) }
\end{equation}
for $y\in \{1,\dots, N\}.$ Note that it is of the form \eqref{product} and thus is definite positive.

\begin{lemma}\label{laplacetrans2}
For all $N$ sufficiently large
for all $\alpha\in \bbR$, $|\alpha|\le \log 2$ and for all $t\ge N^2$, we have 
\begin{equation}\label{criolodoido}
\bbE\left[e^{\alpha \sum_{i=1}^k \left(\ind_{\{X_i(t)\in [1,y]\}}-\bbP\left[X_i(t)\in[0,y] \right]\right)}\right] \le \exp\left( \frac{2ky}{N} \alpha^2\right).
\end{equation}

\end{lemma}
To deduce \eqref{gronek} from \eqref{criolodoido} we just have to remark that 
$$\sum_{i=1}^k \bbP\left[X_i(t)\in[1,y] \right]=\sum_{z=1}^y \sum_{i=1}^k p_t(x^0_i,z)= \sum_{z=1}^y \bbP[\eta_t=1].$$

\begin{proof}[Proof of Lemma \ref{laplacetrans2}]

We using the inequality 
$$\forall |x|\le |\log 2|,\  e^x\le 1+x+x^2$$
for the variable $Z=\alpha \left(\ind_{X_1(t)\in [1,y]}-\bbP[X_1(t)\in [1,y]\right)$.
As $\bbE[Z]=0$ and from Lemma \ref{bound} $(ii)$ we have for all $t\ge N^2$,
$$\bbE[Z^2]\le \bbP\left[X_1(t)\in[1,y] \right] \le 2y/N,$$
the integrated inequality gives
\begin{equation}
\bbE\left[ e^{\alpha \left(\ind_{\{X_1(t)\in [1,y]\}}-\bbP[X_1\in [1,y]\right)} \right]\le  1+\frac{2\alpha^2y}{N}
\le \exp(2\alpha^2 (y/N)).
\end{equation}

The independence of the $X^i$s is then sufficient to conclude.
\end{proof}

\section{Coupling $P_t^\chi$ with the equilibrium using the corner-flip dynamics}

In this section we present the main tool which we use to prove Proposition \ref{csdf}: the corner-flip dynamics.
The idea is to associate to each $\eta\in \gO$ an height function, and consider the dynamics associated with this rate function instead of the original one 
and use monotonicity properties of this latter dynamics. This idea is is already present in the seminal paper of Rost investigating 
the asymetric exclusion on the line
\cite{cf:Rost} and became since a classical tool in the study of particle system. 
In particular it is used e.g.\ in \cite{cf:Wilson, cf:Lac} to obtain bounds on the mixing time for the exclusion on the line. 
It has also been used as a powerful tool for the study of mixing of monotone surfaces starting with \cite{cf:Wilson}, and more recently in \cite{cf:CMST, cf:CMT}.

\medskip

Let us stress however that in \cite{cf:Wilson, cf:Lac}, the interface represention is used mostly as graphical tool in order to have a better intuition on 
an order that can be defined directly on $\gO_{N,k}$.
In the present we use the interface representation to construct a coupling which cannot be constructed considering only the original chain.
In particular, note that our coupling is Markovian for the corner-flip dynamics but not for the underlying particle system.

\label{fluctuat}
\subsection{The $\xi$ dynamics}

Let us consider the set of height functions of the circle.
\begin{equation}
\gO'_{N,k}:=\left\{\xi: \bbZ_N\to \bbR \ | \
\xi(x_0)\in \bbZ, \forall x\in \bbZ_N,\   \xi(x)-\xi(x+1)\in \big\{-\frac k N, 1-\frac k N \big\} \right\}.
\end{equation}

Given $\xi$ in $\gO'_{k,N}$, we define $\xi^x$ as
\begin{equation}
 \begin{cases}
\xi^x(y)&=\xi(y), \quad \forall y \ne x,\\
\xi^x(x)&=\xi(x+1)+\xi(x-1)-2\xi(x).
 \end{cases}
\end{equation}
and we let $\xi_t$ be the irreducible Markov chain on $\gO'_{m,N}$ whose transition rates $p$ are given by 
\begin{equation}\label{cromik}
 \begin{cases}
 p(\xi,\xi^x)&=1, \quad \forall x\in \bbZ_N, \\
 p(\xi,\xi')&=0, \quad \text{if } \xi'\notin \{ \xi^x \ | \ x\in \bbZ_N\}.
 \end{cases}
\end{equation}
We call this dynamics the corner-flip dynamics, as the transition $\xi\to\xi^x$ corresponds to flipping
either a local maximum of $\xi$ (a "corner" for the graph of $\xi$)  to a local minimum \textit{e vice versa}.
It is, of course, not positive recurrent, as the state space is infinite and  the dynamics is left invariant by vertical translation. However
it is irreducible and recurrent.

\medskip

The reader can check (see also Figure \ref{partisys}) that $\gO'_{N,k}$ is mapped onto $\gO_{N,k}$, by the transformation 
$\xi \mapsto \grad \xi$ defined by
\begin{equation}\label{projec}
\grad \xi (x):=\xi(x)-\xi(x-1)+\frac{k}{N},
\end{equation}
and that
the image $\grad \xi_t$ of the corner-flip dynamics $\xi_t$ under this transformation is the simple exclusion,
down and up-flips corresponding to jumps $x\to x+1$ and $x\to x-1$ of the particles respectively.
There is a natural order on the set $\gO'_{N,k}$ 
 defined by 
\begin{equation} \label{deforder}
\xi\ge \xi' \quad \Leftrightarrow \quad \forall x\in \bbZ_N,\ \xi(x)\ge \xi'(x),
\end{equation}
which has the property of being preserved by the dynamics in 	certain sense (see Section \ref{grafff} for more details).

\begin{figure}[hlt]
 \epsfxsize =9.5 cm
 \begin{center}
 \epsfbox{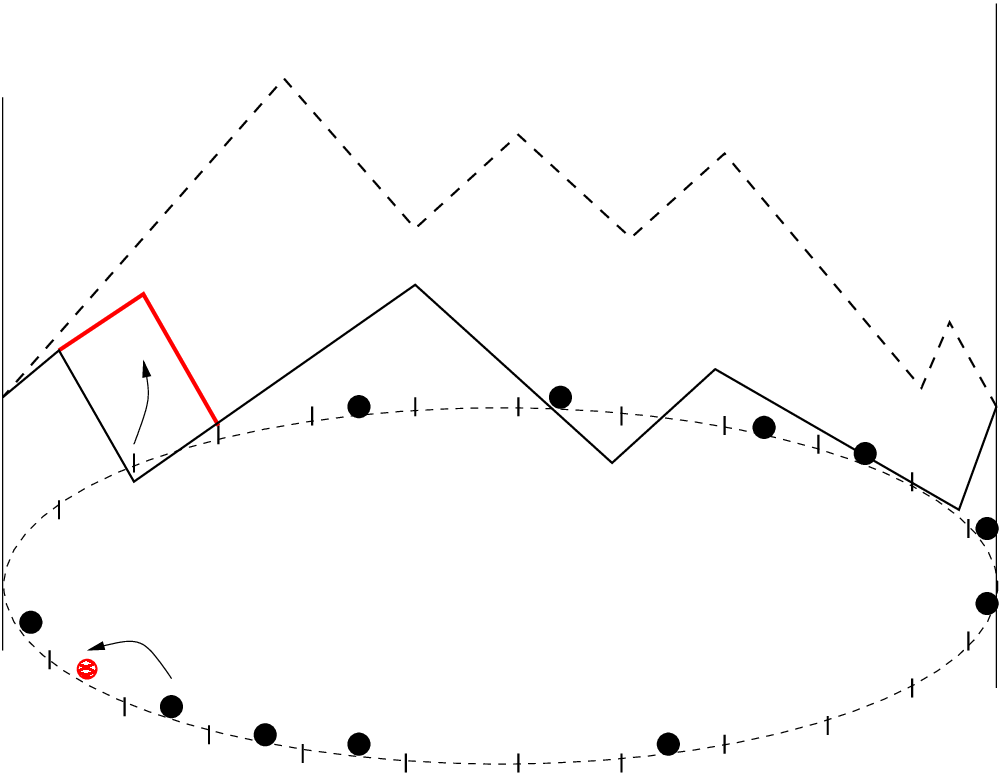}
 \end{center}
 \caption{\label{partisys}  The correspondence between the exclusion process and the corner-flip dynamics.
 In red, a particle jump and its corner-flip counterpart are defined. Note that this is not a one-to-one mapping as a particle configuration gives the height function only modulo translation (the height function is drawn on a cylinder whose base is the circle on which the particles are moving).}
 \end{figure}

Given $\chi\in \gO_{m,k}$, we define $(\xi^0_t)$ to be a process with transitions \eqref{cromik} starting from initial condition
 \begin{equation}\label{fryied}
 \xi_{0}^0(x):=\sum_{z=0}^x \chi (x)-\frac{kx}{N},
\end{equation}
It follows from the above remark that fora all $t\ge 0$ we have
\begin{equation}\label{laloi}
 \bbP\left[ \grad \xi^0_t\in \cdot \right]=P^\chi_{t}.
\end{equation}

Our idea is to construct another dynamic $\xi^1_t$ which starts from a stationary condition (the gradient is distributed according to $\mu$) and to try to couple it with $\xi^0_t$ within time $O(L^2)$.
The difficulty here lies in finding the right coupling.

\subsection{Construction of the initial conditions $\xi^1_0$ and $\xi^2_0$}

In fact we define not one but two stationary dynamics $\xi^1_t$ and $\xi^2_t$, 
satisfying 
\begin{equation}\label{equistart}
\bbP\left[\grad \xi^1_t\in \cdot \right]=\bbP\left[\grad \xi^2_t\in \cdot\right]=\mu.
\end{equation}
As $\mu$ is invariant for the dynamics $\grad \xi_t$, \eqref{equistart} is satisfied for all $t$ as soon 
as it is satisfied for $t=0$.
As we wish to use monotonicity as a tool, we want to have
\begin{equation}\label{katzodue}
\forall t\ge 0, \quad \xi^1_t\le \xi^0_t\le \xi^2_t,
\end{equation}
Then our strategy to couple $\xi^0_t$ with equilibrium is in fact to couple 
$\xi^1_t$ with $\xi^2_t$ and remark that if   \eqref{katzodue}  holds then 

\begin{equation}\label{paninosto}
\forall t\ge 0,\   \xi^1_t=\xi^2_t \ \Rightarrow \ \xi^1_t=\xi^0_t= \xi^2_t
\end{equation}

\medskip

We first have to construct the initial condition $\xi^1_0$ and $\xi^2_0$ which satisfies \eqref{equistart}
\begin{equation}\label{katzone}
\xi^1_0\le \xi^0_0\le \xi^2_0.
\end{equation}

\medskip

Let us start with variable $\eta_0$ which has law $\mu$.
We want to construct $\xi^1_0$ and $\xi^2_0$
which satisfies
\begin{equation}
\grad \xi^i_0=\eta_0.
\end{equation}
Somehow, we also want the vertical distance between $\xi^1_0$ and $\xi^2_0$ to be as small as possible.
We set for arbitrary $\eta \in \gO_{N,k}$, or $\xi\in \gO'_{N,k}$
\begin{equation}\begin{split}\label{defh}
H(\eta)&:=\max_{x,y\in \bbZ_N} \left |\sum_{z=x+1}^y \left(\eta(z)-\frac{k}{N}\right)\right |,\\
H(\xi)&:=\max_{x,y\in \bbZ_N} \left | \xi(x)-\xi(y) \right |.
\end{split}\end{equation}
Finally set we set 
\begin{equation}
\label{defh0}
H_0:=\big\lceil H(\eta_0)+s\sqrt{k} \big\rceil 
\end{equation}
and 

 \begin{equation}\label{fryieed}
 \begin{split}
 \xi_{0}^1(x)&:=\sum_{z=1}^x \eta_0(x)-\frac{kx}{N}-H_0. \\
  \xi_{0}^2(x)&:=\sum_{z=1}^x \eta_0(x)-\frac{kx}{N}+H_0. 
\end{split}
 \end{equation}
Note that with this choice, \eqref{katzone} is satisfied for $\chi \in \mathcal G_s$ (see Figure \ref{highfunk}).

 \begin{figure}[hlt]
 \epsfxsize =8.5 cm
 \begin{center}
 \psfragscanon
  \psfrag{H0}{$2H_0$}
 \psfrag{eta10}{$\xi^2_0$}
  \psfrag{eta20}{$\xi^1_0$}
    \psfrag{eta00}{$\xi^0_0$}
     \epsfbox{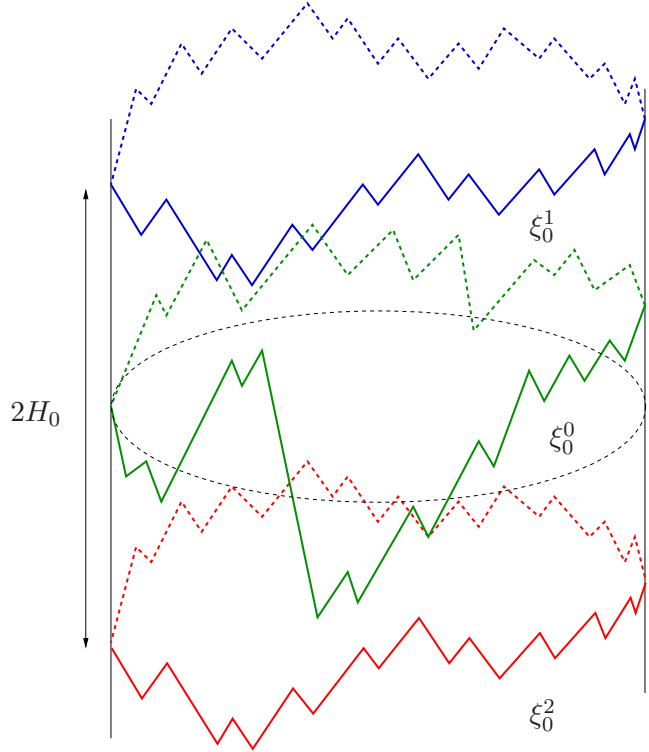}
      \end{center}
 \caption{\label{highfunk} Representation of the three initial condition for the corner-flip dynamics.
 $\xi^1_0$ and $\xi^2_0$ are translated version of the same profile.
 The height $H_0$ is designed so that initially $\xi^0_0$ (whose variation are smaller than $s\sqrt{k}$ if $\chi\in \mathcal G_s$ is framed by  $\xi^1_0$ and $\xi^2_0$).
 As the order is conserved by the graphical construction. $\xi^0_t$ couples with equilibrium when $\xi^1_t=\xi^2_t$.}
 \end{figure}

\subsection{The graphical construction}\label{grafff}

Now we present a coupling which satisfies \eqref{katzodue}.
Note that in this case $\xi^1_t=\xi^0_t=\xi^2_t$ if and only if the area between the two paths, defined by
\begin{equation}\label{paxvolumi}
 A(t):=\sum_{x\in \bbZ_N} \xi^2_t(x)-\xi^1_t(x),
\end{equation}
equals zero.

\medskip

The idea is then to find  among the order-preserving possible one
for which the ``volatility'' of $A(t)$ is the largest possible, so that it reaches zero faster.
We want to make the corner-flips of $\xi^1$ and $\xi^2$ \textsl{as independent as possible}
(of course making them completely independent is not an option since \eqref{katzodue} would not hold)

\medskip

We introduce now our coupling of the $\xi^i_t$, which is also a grand-coupling on $\gO'_{N,k}$,
in the sense that it allows us to construct $\xi_t$ starting from all initial condition on the same probability space.
The evolution of the $(\xi_t)_{t\ge 0}$ is completely determined by auxiliary Poisson processes which we call clock processes.
Set 
$$\Theta:=\left\{ (x,z) \ | \ x\in \bbZ_N \text{ and } z\in \bbZ+\frac{kx}{N} \ \right\}$$

And set $\cT^\uparrow$ and $\cT^\downarrow$ to be two independent rate-one clock processes 
indexed by $\Theta$ ($\cT^\uparrow_{\theta}$ and 
$\cT^\downarrow_{\theta}$ are two independent Poisson processes of intensity one of each $\theta\in \Theta$).
The trajectory of $\xi_t$ given $(\cT^\uparrow,\cT^\downarrow)$ is given by the following construction
\begin{itemize}
 \item $\xi_t$ is a c\`ad-l\`ag, and does not jump until one of the clocks indexed by $(x,\xi_t(x))$, $x\in \bbZ_N$.
 \item If $\cT^\downarrow_{(x,\xi_{t^-}(x))}$ rings at time $t$ and $x$ is a local maximum for $\xi_{t^-}$, then
 $\xi_{t}=\xi^x_{t^-}$.
 \item  If $\cT^\uparrow_{(x,\xi_{t^-}(x))}$ rings at time $t$ and $x$ is a local minimum for $\xi_{t^-}$, then
 $\xi_{t}=\xi^x_{t^-}$.
\end{itemize}
The coupling of $\xi^0_t$, $\xi^1_t$ and $\xi^2_t$ is obtained by using the same clock process for all of them.
The reader can check that with this coupling \eqref{katzodue} is a consequence of \eqref{katzone}.

\begin{rem}
In fact, the corner flip dynamics which is considered here, is in one to one correspondence with the zero-temperature stochastic Ising model on an infinite cylinder.
With this in mind the coupling we have constructed just corresponds to the the graphical construction of this spin flip dynamics.
See e.g.  \cite[Section 2.3 and Figure 3]{cf:LST} for more details for the dynamics on a rectangle with mixed boundary condition. 
\end{rem}

\begin{rem}
Let us stress here that the coupling we use here is not the one the one of \cite{cf:Wilson} or \cite{cf:Rost} for which the updates of pair of neighbors are done simultaneously for the coupled chains. In particular it is not a Markovian coupling for the particle system (as the height function is not encoded in the particle configuration). This is a crucial point here as this is what allows the coupling time to be much shorter. Recall in particular in \cite{cf:Wilson} (see Table 1), it is shown that with the usual Markovian coupling, the coupling time is twice as large as the mixing time.
\end{rem}

To prove Proposition \ref{csdf} it is sufficient to prove that $\xi^1_t$ and $\xi^2_t$
typically merge within a time $O(L^2)$. More precisely,

\begin{proposition}\label{nahnou}
For all $s>0$, given $\gep>0$ there exists $C(\gep,s)$ such that for all
sufficiently large $N$,
\begin{equation}\label{nahnou1}
\bbP[  \xi^1_{CN^2}\ne  \xi^2_{CN^2} ]\le \gep. 
\end{equation}
Similarly for all $s, u>0$, there exists $c(s,u)>0$ such that 
\begin{equation}\label{nahnou2}
\bbP[  \xi^1_{uN^2}\ne  \xi^2_{uN^2} ]\le 1-c(s,u).
\end{equation}
\end{proposition}

\begin{proof}[Proof of Proposition  \ref{csdf}]

Given $s>0$  consider $(\xi^1_t)_{t\le 0}$, $(\xi^2_t)_{t\ge 0}$, constructed as above.
Then given $\chi \in \mathcal G_s$, we construct $\xi^0_t$ the dynamics starting from the initial condition \eqref{fryied} and using the same clock process as $\xi^1_t$ and $\xi^2_t$.
By definition of  $\mathcal G_s$, \eqref{katzone} is satisfied  and thus so is \eqref{katzodue} from the graphical construction.
Recalling \eqref{laloi} and \eqref{equistart} we have
\begin{equation}
\|P^\chi_{t}-\mu\|\le \bbP\left[ \grad \xi^0_{t}\ne \grad \xi^1_{t}\right]\le \bbP\left[ \xi^0_{t}\ne  \xi^1_{t}\right]
\le  \bbP\left[\xi^1_{t}\ne \xi^2_{t}\right].
\end{equation}
for any $t>0$, where the last inequality is a consequence of \eqref{paninosto}.
According to the above inequality, Proposition \ref{csdf} obviously is a consequence of Proposition \ref{nahnou}.

\end{proof}

\section{The proof of Proposition \ref{nahnou}}\label{multis}

In order to facilitate the exposition of the proof, we choose to present it in the case $k=N/2$ first. We also chose to focus on \eqref{nahnou1}.
The necessary modifications to prove \eqref{nahnou2} and to adapt the proofs for general $k$ are explained at the end of the section.

\subsection{The randomly walking area}

We are interested in bounding (recall \eqref{paxvolumi})
\begin{equation}\label{deftau}
\tau :=\inf \{ t \ge 0 \ | \ A(t)=0 \}= \inf \{ t \ge 0 \ | \ \xi^1_t=\xi^2_t \}
\end{equation}
With our construction, $A(t)$, the area between the two curves
is a $\bbZ_+$ valued martingale which only makes nearest neighbor jumps (corners flip one at a time).

\medskip

Hence $A(t)$ is just a time changed symmetric nearest neighbor walk on $\bbZ_+$ which is absorbed at zero.
In order to get a bound for the time at which $A$ hits zero, we need to have reasonable control on the jump rate
which depends on 
the particular configuration $(\xi^1_t,\xi^2_t)$ the system sits on. The jump rate is given by 
the number corners of $\xi^1_t$ and $\xi^2_t$ that can flip separately.
More precisely,
set 

\begin{multline}\label{defui}
U_i(t):=\{x\in \bbZ_N \ | \ \xi^i_t  \text{ has  a local extremum at  } x 
\text{ and } \\
\exists y\in \{x-1,x,x+1\}, \xi^2_t(y)> \xi^1_t(y)\}.
\end{multline}
The jump rate of $A(t)$ is given by 
\begin{equation}
u(t):=\#U_1(t)+\#U_2(t).
\end{equation}
For $t\le \int_0^\tau u(t) \dd t$ let us define
\begin{equation} \label{defjt}
J(t):=\inf\left\{s \ | \  \int_0^s u(v)\dd v\ge t\right\}.
\end{equation}
By construction, the process $(X_t)_{t\ge 0}$ defined by 
\begin{equation}\label{defx}
X_t:=A(J(t))
\end{equation} 
is a continuous time random walk on $\bbZ_+$ which jumps up and down with rate $1/2$.
We have from the definition of the $\xi^i_0$
$$X_0=A(0)=2H_0N$$
which is of order $N^{3/2}$ and hence
$X_t$ needs a time of order $N^{3}$ to reach $0$. 
What we used to estimate $A(0)$ is the following bound which can be derived from \eqref{fluquetec} and the definition of $H_0$ \eqref{defh0},

\begin{equation}\label{grimic}
\bbP\left[A(0) \ge 2(s+r)N^{3/2} \right]=\bbP\left[H_0\ge (s+r)N^{1/2} \right] \le  \bbP\left[H(\eta_0)\ge rN^{1/2} \right]  \le 2 e^{-cr^2}.
\end{equation}

\medskip

If $u(t)$ were of order $N$ for all $t$ this would be sufficient to conclude that $A(t)$ reaches zero within time $O(N^2)$.
This is however not the case: the closer $\xi^1_t$ and $\xi^2_t$ get, the smaller $u(t)$ becomes.
A way out of this is to introduce a multi-scale analysis where the bound we require on $u$ depends on how small
$A(t)$ already is.

\subsection{Multi-scale analysis}

We construct a sequence of intermediate stopping time $(\tau_i)_{i\ge 0}$ as follows.
\begin{equation}
\tau_i:=\inf\left\{t\ge 0 \ | \ A(t)\le N^{3/2}2^{-i} \right\}.
\end{equation}
We are interested in $\tau_i$ for $i\in\{0,\dots,K\}$ with
\begin{equation}\label{defK}
K_N:=\left\lceil \frac{1}{2}\log_2 N \right\rceil.
\end{equation}
Note that a number of $\tau_i$ can be equal to zero if $A(0)\le N^{3/2}$.
We set $\tau_{-1}:=0$ for convenience.

 \medskip
 
 To bound the value of $\tau$, our aim is to control each increments
 $\gD \tau_i=\tau_{i}-\tau_{i-1}$ for $i\le K$ as well $\tau-\tau_K$.
A first step is to get estimates for the equivalent of the $\Delta \tau$ for the time rescaled process 
 $X_t$ (recall \ref{defx}).
We set for $i\in\{0,\dots,K\}$

\begin{equation}\label{freddo}\begin{split}
\cT_i&:=\int_{\tau_{i-1}}^{\tau_i} u(t) \dd t,\\
\cT_\infty&:=\int_{\tau_K}^{\tau} u(t) \dd t.
\end{split}\end{equation}
As $X$ is diffusive, $\cT_i$ is typically of order $(N^{3/2} 2^{-i})^2=N^34^{-i}$, and
$\cT_{\infty}$ is of order $N^2$. With this in mind, it is not too hard to believe that

\begin{lemma}\label{cromican}
Given $\gep, s>0$
there exists a constant $C(\gep,s)$ such that 
\begin{equation}
\bbP[ \left \{\exists i\in \{0,\dots, K\}, \cT_i\ge C N^3 3^{-i}\right\} \cup \{\cT_\infty\ge C N^2 \}]\le \gep.
\end{equation}
\end{lemma}

\begin{proof}

Let $Z_t$ denote a nearest neighbor walk on $\bbZ$ starting from $0$ and $T_a$ the first time $Z$ reaches $a$. It is rather standard that there exists a constant $C_1$ 
such that for every $a\ge 1$ and every $u\ge 0$
\begin{equation}\label{crooop}
\bbP\left[ T_a \ge u a^2\right]\le C_1 u^{-1/2}.
\end{equation}
Note that for $i\ge 1$, ignoring the effect of integer rounding,  $\tau_i$ has the same law as  $T_a$ with $a=N^{3/2}2^{-{i}}$ and thus applying \eqref{crooop} we obtain that
\begin{equation}
\bbP[\cT_i\ge u  N^3 3^{-i}]\le C_1  u^{-1/2} (3/4)^{i/2}.
\end{equation}
In the same manner we have
\begin{equation}
\bbP[\cT_{\infty}\ge u  N^2]\le C_1  u^{-1/2}.
\end{equation}
We can then choose $u_0(\gep)$ large enough in a way that 
\begin{equation}
C_1  u^{-1/2}\left( \sum_{i=1}^K  (3/4)^{i/2}+ 1\right) \le \gep/2.
\end{equation}
Concerning $\cT_0$, from \eqref{grimic}, one can find  $C_2(\gep,s)$ such that 
\begin{equation}
\bbP\left[A(0)\ge C_2 N^{3/2}\right]\le \gep/4.
\end{equation}
Conditionally on the event $A(0)\le C_2 N^{3/2}$, $\cT_0$ is stochastically dominated by $T_a$ with $a=C_2N^{3/2}$ and hence
using \eqref{crooop} and fixing $u_1(\gep,s)$ large enough (depending on $C_1$, $C_2$ and $\gep$) we obtain 

\begin{equation}
\bbP\left[ \cT_0 \ge u_1 (C_2)^2 N^{3}\right]\le \bbP\left[A(0)\ge C_2 N^{3/2}\right]+\bbP\left[ \cT_0 \ge u A(0)^2\right]\ge \gep/2.
\end{equation}
Then we conclude by taking $C(\gep,s):=\max(u_0,u_1 (C_2)^2)$.

\end{proof}

What we have to check then is that the value of $u(t)$ is not too small in the time interval 
$[\tau_{i-1},\tau_i)$ for all $i\in \{0,\dots, K\}$. What we want to use is that for any $t\ge 0$, $\grad \xi^1_t$ is at equilibrium
so that $\xi^1_t$ has to present a ``density of flippable corners''.
We introduce an event $\mathcal A$ which is aimed to materialize this fact.
Given $x$ and $y$ in $\bbZ_N$ we set 

\begin{equation}\label{defj}
j(x,y,\xi):=\#\{ z\in [x,y] \ | \ \xi(z) \text{ is a local extremum } \}.
\end{equation}
We have
\begin{equation}
\mu\left(j(x,y,\xi)\right)=\frac{(N-2)\#[x,y]}{2(N-1)}.
\end{equation}
We define
\begin{equation}\label{defca}
\mathcal A:= \Big\{ \forall t\le N^3,  \forall (x,y) \in \bbZ_N^2,\ \#[x,y]\ge N^{1/4}\Rightarrow j(x,y,\xi^1_t)\le \frac{1}{3}\#[x,y] \Big\},
\end{equation}
the event that a ``large" interval with an anomalously low density of corner does not appear before time $N^3$.

\begin{lemma}\label{smalla}
For all $N$ sufficiently large
\begin{equation}
 \bbP[\mathcal A^c] \le \frac{1}{N}.
\end{equation}
\end{lemma}

\begin{proof}[Proof of Lemma \ref{smalla}] 
 
Note that for any given time $t$, $\grad \xi^1_t$ is distributed according to 
$\mu$ (because this is the case for $t=0$ and $\mu$ is the equilibrium measure for the $\grad \xi$ dynamics).
Now let us estimate the probability of 
$$\cE:=\left\{\eta \ | \ \forall x,y \in \bbZ_N,\ \#[x,y]\ge N^{1/4}\Rightarrow j(x,y,\eta)\le \frac{1}{3}\#[x,y]\right\},$$
under the measure $\mu$ where $j(x,y,\eta)$ is defined like its counter part for the height function \eqref{defj}
replacing ``$x$ is a local extremum" by "$\eta(x)\ne \eta(x+1)$".

We consider $\tilde \mu$ an alternative measure on $\{0,1\}^{\bbZ_N}$, under which the $\eta(x)$ are i.i.d.\ Bernoulli random variable with parameter $1/2$.
By the local central limit Theorem for the random walk we have

\begin{equation}\label{groom}
\mu(\cE):= \tilde\mu\left(\cE  \ | \ \sum_{x\in \bbZ_N} \eta(x)=N/2\right)\le C_1\sqrt{N} \tilde \mu (\cE).
\end{equation}
Let us now estimate  $\tilde \mu (\cE).$
First we remark that we can replace ``$ \#[x,y]\ge N^{1/4}$" in the definition of $\cE$ by ``$\#[x,y] \in [N^{1/4}, 2N^{1/4}]$".
Indeed,  by dichotomy, if the proportion of local maxima is smaller than $1/3$ on a long interval, it has to be smaller than $1/3$ on a subinterval whose length belong to $[N^{1/4}, 2N^{1/4}]$.

\medskip

Set $x\in \{ \lceil N^{1/4} \rceil, N-1\}$, note that  $\ind_{\{\eta(z)\ne \eta(z+1)\}}$, $z\in \{1,\dots, x\}$ are IID Bernouilli variables of parameter $1/2$, and hence 
by standard large deviation results there exists a constant $C_2>0$ such that
\begin{equation}\label{acroot}
\tilde \mu\left( \sum_{z=1}^x \ind_{\{\eta(z)\ne \eta(z+1)\}} \le N/3 \right)\le e^{-C_2x}
\end{equation}
By translation invariance we can deduce similar bounds for any translation of the interval $[1,x]$. 
Then, summing over all intervals and using \eqref{groom}, we deduce that there exists $C_3$ such that, for all $N$ sufficiently large
 $\mu(\cE)\le e^{-cN^{1/4}}.$

Now, we  set $(T_i)_{i\ge 0}$ to be the times where the chain $\xi^2$ makes a transition. 
 The chain $(\xi^1_{T_i})_{i\ge 0}$ is a discrete time Markov chain with equilibrium probability $\mu$ and hence by union bound

 \begin{equation}
 \bbP\left[\exists t\le T_i \ \grad \xi^1_t(t)\notin \cE \right] \le i e^{-C_3N^{1/4}}.
 \end{equation}
This implies
  \begin{equation}
\bbP[\cA] = \bbP\left[\exists t\le N^3 \grad \xi^1_t(t) \notin \cE \right] \le i e^{-C_3N^{1/4}}+\bbP\left[ T_i\le N^3\right].
 \end{equation}
 
As the transitions occur with a rate which is at most $N$,  the second term is exponentially small e.g. for $i=N^5$ and this concludes the proof.
 \end{proof}
 
Then when $\mathcal A$ holds, we can derive an efficient lower bound on $u(t)$ which just depend on $A(t)$.
Recall \eqref{defh}

\begin{lemma}\label{fromzea}
When $\mathcal A$ holds we have for all $t\le N^3$
\begin{equation}
u(t)\ge \frac 1 3 \min\left( N, \frac{A(t)}{\max(H(\xi^1_t)+H(\xi^2_t), N^{1/2})}\right)
\end{equation}
\end{lemma}

\begin{proof}
 
If $\xi^1_t$ and $\xi^2_t$ have no contact with each other, then 
$u(t)$ is equal to the total  number of flippable corners in $\xi^2_t$ and $\xi^1_t$.
If $\cA$ holds, this number is larger than $N/3$, which, by definition of $\cA$, is a lower bound for the number of corners on $\xi^1_t$ alone.
When there exists $x$ such that $\xi^1_t(x)=\xi^2_t(x)$, we consider the set of active coordinates 
\begin{equation}\label{defct}
C(t):=\{\exists y\in \{x-1,x,x+1\}, \xi^1_t(y)< \xi^2_t(y)\}.
\end{equation}
Note that when one of the $\xi^i_t$ (or both) have a local maximum at $x\in C(t)$ then when  the corresponding corner flips,f it changes the value of $A(t)$.
Our idea is to find a way to bound from below the number of $x$ in $C(t)$ for which $\xi^1(t)$ has a flippable corner, using the assumption that $\cA$ holds.

\medskip

Let us decompose $C(t)$ into connected components (for the graph $\bbZ_N$) which are intervals as defined in \eqref{theinterval}.
Assume that $[a,b]$ is a connected component of $C(t)$, it corresponds to a 
``bubble'' between $\xi^1_t$ and $\xi^2_t$ (see Figure \ref{bubble}).
For each bubble, we want to have a bound on the number of flippable corners and compare it to the 
area of the bubble.
Set 
\begin{equation}\label{defuab}
u_{[a,b]}(t):=j(a,b,\xi^1_t).
 \end{equation}
 and 
 \begin{equation}\label{bubarea}
A_{[a,b]}(t):=\sum_{x=a}^{b}\xi^2_t(x)-\xi^{1}_t(x).
\end{equation}
 Note that 
 \begin{equation}\label{lesum}\begin{split}
 u(t)&\ge \sum_{\text{ all bubbles}} u_{[a,b]}(t),\\
A(t)&= \sum_{\text{ all bubbles}} A_{[a,b]}(t).
 \end{split}\end{equation}

 \begin{figure}[hlt]
 \epsfxsize =12.5 cm
 \begin{center}
 \psfragscanon
 \psfrag{etw}{$\xi^2_t$}
  \psfrag{etav}{$\xi^1_t$}
    \psfrag{buble}{$[a,b]$}
     \epsfbox{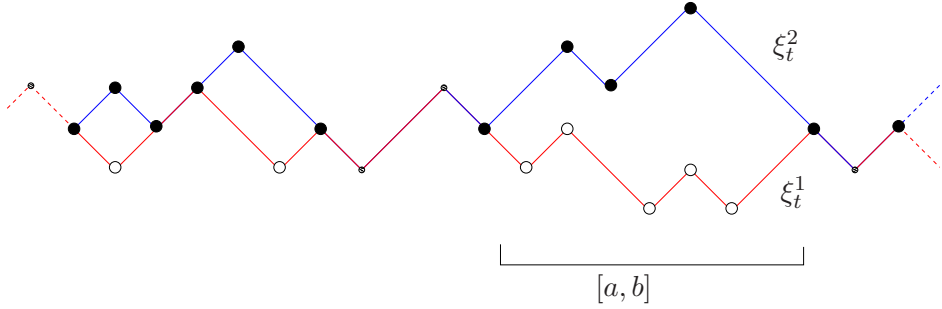}
      \end{center}
 \caption{\label{bubble} The bubble decomposition: The interval $[a,b]$ displayed here corresponds to a bubble.
  The large circle dots corresponds to corner which do not flip simultaneously for $\xi^1_t$ and $\xi^2_t$, the total number of them is $u(t)$.
 Among those, the white circles are the one which are counted in one of the $u_{[a,b]}(t)$ (note that some of the corners on $\xi^1_t$ are not counted). The smaller circles corresponds to corner which flip together for $\xi^1_t$ and $\xi^2_t$ and thus do not contribute to $u(t)$.}
 \end{figure}

For small bubbles (of length smaller than $N^{1/4}$), $\mathcal A$ does not give any information on the number of flippable corners. 
However, we can simply observe that in any bubble, there is at least one flippable corners (e.g. where $\min_{x\in [a,b]}\xi^1_x$ is attained).
If $\#[a,b] \le N^{1/4}$, the area of the bubble satisfies
\begin{equation}
A_{[a,b]}(t)\le N^{1/2}.
\end{equation}
This is because $\xi^2_t-\xi^{1}_t$ is a Lipchitz function and equals zero at both ends. Hence necessarily
\begin{equation}\label{esquiun}
 u_{[a,b]}(t)\ge \frac{A_{[a,b]}(t)}{N^{1/2}}.
\end{equation}
For large bubbles ($\#[a,b]\ge N^{1/4}$) we can use the fact that $\cA$ hold.
First, let us control the area:
as $\xi^1_t$ and $\xi^2_t$ are in contact we have 
\begin{equation}
\max_{x\in \bbZ_N}  \xi^2_t(x)- \xi^1_t(x)\le  H(\xi^1_t)+H(\xi^2_t),
\end{equation}
and hence 
\begin{equation}
A_{[a,b]}(t)\le \#[a,b] (H(\xi^1_t)+H(\xi^2_t)).
\end{equation}
By the definition of $\cA$ there are at least $\#[a,b]/3$ flippable corners on the path $\xi^1_t$ restricted to $[a,b]$.
Thus
\begin{equation}
 u_{[a,b]}(t)\ge \frac{A_{a,b}(t)}{3(
H(\xi^1_t)+H(\xi^2_t))}.
\end{equation}
and we can deduce (recall \eqref{esquiun}) that for any value $\#[a,b]$ we have
\begin{equation}\label{esqui2}
 u_{a,b}(t)\ge \frac{A_{a,b}(t)}{3\max(H(\xi^1_t)+H(\xi^2_t),\sqrt{N)})}.
\end{equation}
We conclude by summing \eqref{esqui2} over all bubbles and using \eqref{lesum}.
\end{proof}

The previous lemma gives us some control over $u(t)$ (if $\cA$ holds)
provided we can control $A(T)$ and 
\begin{equation}\label{defht}
H(t):=H(\xi^1_t)+H(\xi^2_t).
\end{equation}
To control the area, we use our multi-scale construction: for $t\in [\tau_{i-1}, \tau_i)$, we have
$$A(t)\ge N^{3/2} 2^{-i}.$$ 
To obtain a good control on $H(t)$ we use the following concentration result.

\begin{lemma}\label{cromeski}
There exists a constant $c$ such that for any $t\ge 0$ and $r\ge 0$.
\begin{equation}
\bbP\left[ H(t)\ge r \sqrt{N} \right]\le 2\exp(-c r^2). 
\end{equation}
\end{lemma}
\begin{proof}
This just comes from the fact that for any $t>0$, $i=1,2$, $\grad \xi^i_t$ are distributed according to $\mu$, and then we use \eqref{fluquetec}.
\end{proof}

We use it to show that most of the time $H(t)$ is of order $\sqrt{N}$.
In fact we need a sightly more twisted statement that fits the multi-scale analysis.
For the remainder of the proof set 
\begin{equation}
\alpha:=\left(\sum_{i\ge 0} (i+1)^2 \right)^{-1}
\end{equation}

\begin{lemma}\label{coniduam}
For any $\gep>0$, there exists a constant $C(\gep)$ such that 
for any $T\ge 0$

\begin{equation}
\bbP\left[ \exists i\in \{0,\dots, K\}, \int^T_0 \ind_{\{ H(t)\ge C (i+1)^2 \sqrt{N}\}}\dd t \ge (\alpha/2) (i+1)^{-2}T \right] \le \gep.
\end{equation}
\end{lemma}

\begin{proof}
For a fixed $i$ from Lemma \ref{cromeski}, we have

\begin{equation}
\bbE\left[ \int^T_0 \ind_{\{H(t)\ge r(i+1)^2 \sqrt{N}\}} \dd t \right] \le 2Te^{-cr^2(i+1)^2}. 
\end{equation}
Hence from the Markov inequality 
\begin{equation}
\bbE\left[ \int^T_0 \ind_{\{H(t)\ge r(i+1)^2 \sqrt{N}\}}\dd t\le (\alpha/2)  (i+1)^{-2}T \right] \le 4\alpha^{-1}(i+1)^2 \exp(-cr^2(i+1)^2).
\end{equation}
Hence we obtain the result by choosing $C=r_0$ sufficiently large so that 
\begin{equation}
4 \sum_{i\ge 0} \alpha^{-1}(i+1)^2 \exp(-cr^2(i+1)^2)\le \gep.
\end{equation}
\end{proof}
Combining Lemma \ref{cromican}, \ref{smalla}, \ref{fromzea}, and \ref{coniduam} we can now conclude.

\begin{proof}[Proof of \eqref{nahnou1}]

Let $\gep$ be fixed.  
We fix a constant $C_1(\gep)$  such that Lemma \ref{coniduam} holds for $\gep/3$ instead of $\gep$.
$C_2(\gep,s)$ is chosen so that Lemma \ref{cromican} holds for $\gep/3$.
We define the events

\begin{equation}
\begin{split}
\mathcal B&:= \left\{ \forall i\in \{0,\dots, K\}, \int^{T}_0 \ind_{\{ H(t)\ge C_1 (i+1)^2 \sqrt{N}\}}\dd t\le(\alpha/2) (i+1)^{-2}T \right\},\\
\mathcal C&:=\left \{\forall i\in \{0,\dots, K\}, \cT_i\le C_2 N^3 3^{-i}\right\} \cap \{\cT_\infty\le  C_2 N^2 \},
\end{split}
\end{equation}
where 
\begin{equation}\begin{split}
T&:=6 N^2 C_1 C_2 \alpha^{-1} \max_{i\ge 0}\left[ (i+1)^4(2/3)^i\right],\\
T'&:=T+C_2N^2.
\end{split}\end{equation}

We assume also that $N$ is large enough so that $\cA$ holds with probability larger than $1-\gep/3$ (cf. Lemma \ref{fromzea}).
Hence we have 
$$\bbP[\cA\cap\cB\cap \cC] \ge 1-\gep$$
Now what remains to prove is that 

\begin{equation}\label{letrucs} 
\{\cA\cap\cB\cap \cC\}\subset \{\tau \le T'\}.
\end{equation}
This implies \eqref{nahnou1},
with 
\begin{equation}
C(\gep,s)=T'/N^2=C_2 \left( 6C_1\alpha^{-1} \max_{i\ge 0}\left[ (i+1)^4(2/3)^i\right]+1\right).
\end{equation}
We split the proof of \eqref{letrucs} in two statements.
We want to show first  that on the event $\cA\cap \cB\cap \cC$,
\begin{equation}\label{ige1}
\tau-\tau_K\le C_2 N^2.
\end{equation}
and then that 
\begin{equation}\label{ige2}
\forall i\in \{0,\dots, K\},  \quad  (\tau_{i}-\tau_{i-1})\le C_3(i+1)^{-2}N^2 ,
\end{equation}
where 
$$C_3:=6 C_2C_1\max_{i\ge 0}\left[ (i+1)^4(2/3)^i\right].$$
Combined these inequalities, we have
\begin{equation}
\tau\le C_2N^2+\sum_{i=0}^K C_3(i+1)^{-2}N^2\le T'.
\end{equation}
Note that the \eqref{ige1} is an immediate consequence of $\cC$ as 
\begin{equation}
\cT_\infty=\int_{\tau_K}^\tau u(t)\dd t\ge \tau-\tau_K.
\end{equation}
Let us turn to \eqref{ige2}.
Let us assume that the statement is false and let $i_0$ 
denote the smallest $i$ such that
$$(\tau_{i}-\tau_{i-1})> C_3 (i+1)^{-2} N^2.$$
The definition of $i_0$ (the fact that it is the smallest) implies that 

\begin{equation}\label{bet}
\tau_{i_0-1}+ C_3(i_0+1)^{-2} N^2\le T.
\end{equation}
From $\cB$ we have (using \eqref{bet} inequality for the second inequality)

\begin{multline}\label{cardrive}
 \int_{\tau_{i_0-1}}^{\tau_{i_0}} \ind_{\{H(t)\le C_1 (1+i_0)^2 \sqrt{N}\}}\dd t\\
 \ge \int_{\tau_{i_0-1}}^{\tau_{{i_0}}+ C_3 (i_0+1)^{-2} N^2} \ind_{\{ H(t)\le C_1 (1+i_0)^2 \sqrt{N}\}}\dd t \\
 \ge C_3 (i_0+1)^{-2} N^2-\int_0^T  \ind_{\{ H(t)\ge C_1 (1+i_0)^2 \sqrt{N} \}}\dd t\\
 \ge C_3 (i_0+1)^{-2} N^2-(\alpha/2) T(i_0+1)^{-2}\ge  \frac{1}{2} C_3(i_0+1)^{-2} N^2.
 \end{multline}
For all $t\le \tau_{i_0}$, we have $A(t)\ge N^{3/2} 2^{-i_0}$ by definition and thus
from Lemma \ref{fromzea} and the assumption that $\cA$ holds, we have
\begin{equation}
u(t)\ge \frac 1 3 \min\left(N, \frac{A(t)}{\max(H(t),N^{1/2})}\right)\ge\frac{N^{3/2} 2^{-i_0}}{3C_1 (i_0+1)^2 \sqrt{N}}\ind_{\{H(t)\le C_1 (i_0+1)^2 \sqrt{N}\}}.
\end{equation}
Hence from \eqref{cardrive}
\begin{multline}
\cT_{i_0}\ge  \frac{N^{3/2} 2^{-i_0}}{C_1  (i_0+1)^2 \sqrt{N}} \int_{\tau_{i_0-1}}^{\tau_{i_0}} \ind_{\{H(t)\le C_1 (i_0+1)^2 \sqrt{N}\}}\dd t\\
\ge \frac{1}{6 C_1} C_3 N^3 2^{i_0}(i_0+1)^{-4}\ge 3^{i_0}C_2 N^2,
\end{multline}
where the last inequality comes from the definition of $C_3$.
This brings a contradiction to the fact that $\cC$ holds and ends the proof of \eqref{ige1}.

\end{proof}

\subsection{Proof of \eqref{nahnou2}}

We want to prove now that starting from $\chi\in \mathcal G_s$,
we get significantly closer to equilibrium after a time $uN^2$.
By mononicity in $u$ and $s$ we can restrict to the case where $u=2s^{-1}$ (to avoid using too many letters) and assume that $s$ is sufficiently large.
The elements of the proof are essentially the same that for \eqref{nahnou1} but we have to be more careful.
Instead of Lemma \ref{cromican} we have to prove the following statement
\begin{lemma}\label{cromican3}
There exists a constant $C$ such that for all $s$ sufficiently large 
\begin{equation}
\bbP\left[ \left \{\forall  i\in \{0,\dots, K\}, \cT_i\le s^{-7} N^3 3^{-i}\right\} \cap \{\cT_\infty\le  s^{-1}N^2 \}\right]\ge e^{-Cs^{9}}
\end{equation}
\end{lemma}
\begin{proof}
A first observation is that, by the Markov property for the random walk $X_t=A(J(t))$ (recall \eqref{defx}), the $\cT_i$ are independent. 
To evaluate $\bbP\left[ \cT_i\le s^{-6} N^3 3^{-i}\right]$,
we are going to use a classical estimate of first hitting time of a given level for a simple symmetric random walk:  
the there exists a constant $C_1$ such that for all $a\ge 1$, for all $v\le a$ one has (using the notation of \eqref{cromican})
\begin{equation}\label{crooops}
\bbP\left[ T_{a} \le a^{2}/v \right]\ge e^{-C_1 \max(v^{1/2},v)},
\end{equation}
(it is sufficient to check the estimate when $v$ is large as for $v$ close to zero it is just equivalent to \eqref{crooop}).
The time $\cT_\infty$ is stochastically dominated by $T_N$ and thus
\begin{equation}\label{triton}
\bbP[\cT_\infty\le  s^{-1}N^2]\ge e^{-C_1 s}.
\end{equation}
Neglecting the effect of integer rounding, for $i\ge 0$, $\cT_i$ is equal in law to $T_{2^{-i}N\sqrt{k}}$.
Hence from \eqref{crooop} we have 
\begin{equation}\label{gronac}
\bbP[ \cT_i\ge s^{-7} N^3 3^{-i} ]\ge e^{-C_1 \max((4/3)^{-i/2}s^{7/2},(4/3)^{-i}s^{7}))}.
\end{equation}
Then note that there exists a constant $C_2$ such that for all  $s\ge 1$
\begin{equation}\label{okafor}
C_1\sum_{i=1}^K \max\left((4/3)^{-i/2}s^{7/2},(4/3)^{-i}s^{7}\right)\ge C_2s^7.
\end{equation} 
For $i=0$, $\cT_0$ depends on initial value of the area.
Now let us note that from \eqref{grimic} we have for $s$ sufficiently large
\begin{equation}\label{grima}
\bbP\left[A(0)\le 4s N \right]\ge1- 2e^{-cs^2}\ge 1/2.
\end{equation}
Then conditioned on the event $\left\{ A(0)\le 4sN \right\}$, $\cT_0$ is dominated by 
$T_{4sN}$ and hence

\begin{equation}
\bbP\left[\cT_0\ge s^{-7} N^3 \ | \ A(0)\le 4s N   \right]\ge \exp\left(-16 C_1 s^{9}\right).
\end{equation}
Combined with  \eqref{grima}, this gives us
\begin{equation}\label{gronic}
\bbP\left[\cT_0\ge s^{-4} N^3 \right]\ge e^{-16C_1 s^{9}}/2.
\end{equation}
Using the independence and multiplying the inequalities \eqref{gronic}, \eqref{gronac} and \eqref{triton}  (and using \eqref{okafor}) we obtain the result for some appropriate $C$.
\end{proof}

We also need an estimate on the probability that $H(t)$ (recall \eqref{defht}) is too large which slightly differs from  Lemma \ref{coniduam}
\begin{lemma}\label{coniduam2}
Recall $\alpha:=(\sum_{i\ge 0}(i+1)^{-2})^{-1}$. There exists a constant $C$ such that for all $s$ sufficiently large and all $T$ 

\begin{equation}
\bbP\left[ \exists i\in \{0,\dots, K\}, \int^{T}_0 \ind_{\{ H(t)\ge s^{5} (i+1)^{2} \sqrt{N}\}}\dd t \ge  (\alpha/2)(i+1)^{-2}T \right] \le e^{-C s^{10}}.
\end{equation}
\end{lemma}
\begin{proof}
As for Lemma \ref{coniduam} we just use Lemma \ref{cromeski} and the Markov inequality.
\end{proof}

\begin{proof}[Proof of \eqref{nahnou2}]

Set $T:=s^{-1}N^2$ and
\begin{equation}
\begin{split}
\mathcal B'&:= \left\{ \forall i\in \{0,\dots, K\}, \int^{T}_0 \ind_{\{ H(t)\ge  s^{5} (i+1)^{2} \sqrt{N}\}}\dd t\le (\alpha/2) (i+1)^{-2}T \right\},\\
\mathcal C'&:=\left \{\forall i\in \{0,\dots, K\}, \cT_i\le s^{-7} N^3 3^{-i}\right\} \cap \{\cT_\infty\le N^2s^{-1} \},
\end{split}
\end{equation}

Then from Lemma \ref{cromican}, \ref{smalla}, and \ref{coniduam} we have for $s$ sufficiently large, for $N$ large enough (depending on $s$)
\begin{equation}
\bbP\left[\cA\cap \cB'\cap \cC'  \right]\ge \exp(-C_1 s^9)/2.
\end{equation}
What remains to prove is that $\tau \le 2N^2$ on the event $\cA\cap \cB'\cap \cC'$.
First, we notice that from the definition of $\cT_\infty$ \eqref{freddo}, $\cC'$ readily implies that 
$$\tau-\tau_K\le \cT_\infty\le  N^2 s^{-1}$$
Hence to conclude we need to show that
\begin{equation}
\forall i\in \{0,\dots,K\}, \quad   \tau_i-\tau_{i-1}\le \alpha (i+0)^{-2} T.
\end{equation}
Assume the statement is false  and let $i_0$ be the smallest index such that it is not satisfied.
Using Lemma \ref{fromzea} we have
\begin{multline}
 \cT_{i_0}=\int_{\tau_{i_0-1}}^{\tau_{i_0}} u(t)\dd t
 \ge  \int_{\tau_{i_0-1}}^{\tau_{i_0}} \frac 1 3  \min\left( N , \frac{A(t)}{\max(H(t),N^{1/2})}\dd t\right)\\
\ge  \frac{2^{-i_0} N}{3s^{5} (i+1)^{2}} \int_{\tau_{i_0-1}}^{\tau_{i_0}} \ind_{\{H(t)\le s^{5} (i+1)^{2} \sqrt{N}\}}\dd t.
\end{multline}
Then one has from the definition of $i_0$
$$\tau_{i_0-1}+  \alpha (i+0)^{-2}\le N^2.$$
Hence from $\cC'$
\begin{multline}
 \int_{\tau_{i_0-1}}^{\tau_{i_0}} \ind_{\{H(t)\le s^{5} (i+1)^{2} \sqrt{N}\}}\dd t\\
 \ge \int_{\tau_{i_0-1}}^{\tau_{i_0-1}+\alpha (i_0+1)^{-2} T} \ind_{\{ H(t)\le s^5 (1+i_0)^2 \sqrt{N}\}}\dd t \\
 \ge \alpha (i_0+1)^{-2} N^2-\int_0^{T}  \ind_{\{ H(t)\ge s^5 (1+i_0)^2 \sqrt{N} \}}\dd t\\
 \ge (\alpha/2)(i_0+1)^{-2} T,
 \end{multline}
 and thus (recall $T=N^2s^{-1}$)
 \begin{equation}
 \cT_{i_0}\ge (\alpha/2)2^{-i_0}(i_0+1)^{2} N^3s^{-6}\ge  C_2 3^{-i_0} N^3 s^{-6}.
 \end{equation}
for an explicit $C_2$. This brings a contradiction to $\cB'$ if  $s$ is chosen sufficiently large.

\end{proof}

 \subsection{Proof of Proposition \ref{nahnou} for arbitrary $k$}
 
 The overall strategy is roughly the same, except that we start with an area which is of order $k^{1/2}N$. Hence most 
 of the modifications in the proof can be performed just writing $\sqrt{k}$ instead of $N^{1/2}$. However 
 Lemma \ref{smalla} does not hold for small values of $k$ and one 
 needs a deeper change there. We define the $(\tau_i)_{i\ge 0}$ as follows ($\tau_{-1}:=0$)
\begin{equation}
\tau_i:=\inf\{t\ge 0 \ | \ A(t)\le  k^{1/2}N2^{-i} \}.
\end{equation}
and we set
\begin{equation}\label{defK2}
K_N:=\left\lceil \frac{1}{2}\log_2 k \right\rceil.
\end{equation}
Note that a number of $\tau_i$ can be equal to zero if $A(0)\le N^{3/2}$.

 \medskip

The time changed version $\cT_i$, $i\in \{0,\dots,K\}\cup\{\infty\}$ of $\gD\tau_i$ are defined as in \eqref{freddo}.
We first write down how Lemma \ref{cromican}-\ref{cromican3} and \ref{coniduam}-\ref{coniduam2} can be reformulated in the context of $k$ particles
(the proofs are exactly the same and thus are not included).

\begin{lemma}\label{cromican2}
Given $\gep$
there exists a constant $C_1(\gep,s)$ such that 
\begin{equation}\label{cranchk}
\bbP\left[ \left \{\exists i\in \{0,\dots, K\}, \cT_i\ge C_1 N^2 k 3^{-i}\right\} \cup \{\cT_\infty> C_1 N^2 \}\right]\le \gep.
\end{equation}
and a constant $C_2$ independent of the parameters
\begin{equation}\label{cranchk2}
\bbP\left[ \left \{\exists i\in \{0,\dots, K\}, \cT_i\ge N^2 k s^{-6}\right\} \cup \{\cT_\infty> N^2 \}\right]\ge  \exp\left(-C_2s^{8}\right).
\end{equation}
\end{lemma}

\begin{lemma}\label{coniduam3}
Recall $\alpha:=(\sum_{i\ge 0} (i_0+1)^{-2})^{-1}.$
For any $\gep>0$, exists a constant $C_1(\gep)$ such that 
for any $T\ge 0$

\begin{equation}\label{crimouk}
\bbP\left[ \exists i\in \{0,\dots, K\}, \int^T_0 \ind_{\{ H(t)\ge C_1 (4/3)^i \sqrt{k}\}}\dd t \ge (\alpha/4) (i+1)^{-2}T \right] \le \gep.
\end{equation}
Moreover there exist a constant $C_2$ such that for any $T\ge 0$,
\begin{equation}\label{crimouk2}
\bbP\left[ \exists i\in \{0,\dots, K\}, \int^{T}_0 \ind_{\{ H(t)\ge s^{5}(4/3)^i \sqrt{k}\}}\dd t \ge (\alpha/4) (i+1)^{-2}N^2 \right]\dd t \le e^{-C_2 s^{10}}.
\end{equation}

\end{lemma}
A significant modification is however needed for Lemma \ref{smalla}, as for small $k$, 
we cannot define an event similar to \eqref{defca} which holds with high probability.
Set (recall \eqref{defj})
\begin{multline}
\mathfrak{A}:=\left\{ \xi \in \gO_{k,N} \ | \ \#[x,y]\ge N/k (\log k)^2 \Rightarrow j(x,y,\xi)\ge \frac{k}{10N}\#[x,y]\right\}\\
\cup \left\{ \xi \ | \ \#[x,y]\le N/k (\log k)^2  \Rightarrow |\xi(x)-\xi(y)|\le (\log k)^4\right\}=:\mathfrak{A}_1\cup \mathfrak{A}_2.
\end{multline}
Note that $\kA$ is invariant by vertical translation of $\xi$ and thus only depends on $\grad \xi$.
Hence we can (improperly) consider it as a subset of $\gO_{N,k}$.
 \begin{lemma}\label{smalla2}
We have
\begin{equation}
\mu(\mathfrak{A})\le \frac{1}{k^2},
\end{equation}
as a consequence one has for every $T\ge 0$
\begin{equation}
\bbP\left[ \int^T_0 \ind_{\{ \xi_1(t)\in \mathfrak A \}}\dd t \ge (T/k) \right]\le 1/k
\end{equation}
\end{lemma}

 \begin{proof}
In this proof it is somehow easier to work with the particle system, hence we let $\mu$ be the uniform measure on $\gO_{N,k}$.
We consider $\tilde \mu$ an alternative measure on $\{0,1\}^{\bbZ_N}$, under which the $\eta(x)$ are i.i.d.\ Bernoulli random variable with parameter $k/N$.
From the Local Central Limit Theorem (which in this simple case can be proved  using the Stirling Formula), there exists a constant $C_1$ such that for all choices of $k$ and $N$

\begin{equation}
\tilde \mu\left(\sum_{x\in \bbZ_N}\eta(x)=k \right)\ge \frac{1}{C_1\sqrt k}.
\end{equation}
Hence for any event $A\subset \{0,1\}^{\bbZ_N}$, we have

\begin{equation}
\mu(A)=\tilde \mu\left(A \ | \ \sum_{x\in \bbZ_N}\eta(x)=k \right)\le C_1\sqrt k \tilde \mu(A).
\end{equation}
Hence to prove the result, we just have to prove a slightly stronger upper-bound for the probability $\tilde \mu(\kA)$.
We start proving that 
\begin{equation}
\tilde \mu(\kA_2)\le \frac{1}{k^{3}}.
\end{equation}
In terms of particle,  $\kA_2$ holds
any interval of length smaller than $N/k (\log k)^2$ contains at most $(\log k)^4$ particles.
Set 
$$m_{k,N}=m:= \lceil N/k (\log k)^2 \rceil.$$
It is a standard large deviation computation (computing the Laplace transform and using the Markov inequality) to show that there exists a constant $c$ such that

\begin{equation}\label{gramnic}
\tilde \mu\left( \sum_{x=1}^m \eta(x)\ge \frac{(\log k)^4}{2}\right)\le \exp(- c (\log k)^6).
\end{equation}

Hence by translation invariance, the probability that there exists an interval of the form  $[(i-1)m+1,im]$, $i\in \{1,\dots \lceil N/m \rceil+1\}$
which contains at most $(\log k)^4/2$ particles is smaller than 
$$2 k (\log k)^2 \exp(- c (\log k)^6)\le k^{-2}.$$
As every interval of length smaller than $m$ is included in in the union of at most two intervals of the type $[(i-1)m+1,im]$ we have

$$\tilde \mu(\kA_2) \le \tilde \mu\left( \exists i \in \{1,\dots \lceil N/m \rceil+1\}, \sum_{x=(m-1)i+1}^{mi} \eta(x) \ge \frac{(\log k)^4}{2}\right)\le k^{-2}.$$
We now prove 

\begin{equation}
\tilde \mu(\kA_1)\le \frac{1}{k^{3}}.
\end{equation}

In terms of particle, having a local extremum at $x$ just corresponds to $\eta(x)\ne \eta(x+1)$.
Note that if $\cA$ occurs then, by dichotomy, there exists necessarily an interval of length comprised between $m$ and $2m$ in which the density of extrema is smaller than $k/10N$ and hence the total number of extrema in it is smaller than $\frac k {5N} m$.
Setting $m'=\lceil m/2 \rceil$ this interval must include an interval of the type
$[(i-1)m'+1,im']$, with $i\in \{1,\dots \lceil N/m' \rceil\}$, in which there are at most $k/5N m$ extrema.
Hence, noting that  
$$ \frac k {5N} m\le \frac{(\log k)^2}{5} \quad \text{ and } \lceil N/m'  \rceil\le 3 k(\log k)^2,$$
we have 
\begin{multline}\label{gromido}
\tilde \mu(\kA_1)\le \tilde \mu\left( \exists i\in \{1,\dots \lceil N/m' \rceil\} \sum_{x=(i-1)m'+1}^{im'} \ind_{\{\eta(x)\ne \eta(x+1)\}}\ge \frac{(\log k)^2}{5}\right)
\\ \le 3 k(\log k)^2\tilde \mu\left( \sum_{i=1}^{m'} \ind_{\{\eta(x)\ne \eta(x+1)\}}\ge  \frac{(\log k)^2}{5}\right).
\end{multline}
Now we remark that 
$$\tilde\mu\left(\eta(x)\ne \eta(x+1)\right)=\frac{3k(N-k)}{N^2}\ge \frac{k} N.$$
As the variables $\ind_{\{\eta(2x-1)\ne \eta(2x)\}}$ are i.i.d for $x\in \{1, \lceil m'/2 \rceil\}$ we can use standard large deviation techniques 
for sums of i.i.d. variables and obtain that there exists a constant $c$ for which 

\begin{equation}
\tilde \mu\left( \sum_{i=1}^{m'} \ind_{\{\eta(x)\ne \eta(x+1)\}}\ge  \frac{(\log k)^2}{5} \right)\le \tilde \mu\left( \sum_{x=1}^{\lceil m'/2 \rceil} \ind_{\{\eta(2x-1)\ne \eta(2x)\}}\ge  \frac{(\log k)^2}{5}\right) \le 
e^{-c(\log k)^2}.
\end{equation}
This combined with \eqref{gromido} allows us to conclude.

\end{proof}
\begin{lemma}\label{fromzea2}
When $\xi_1(t)\in \kA$ we have 

\begin{equation}
u(t)\ge \frac{1}{10}\min\left(k, \frac{A(t)k}{N \max(H(t), (\log k)^6)}\right)
\end{equation}
\end{lemma}

\begin{proof}
If $\xi^1_t\in \kA$ and $\xi^2_t(x)>\xi^1_t(x)$, for all $x\in \bbZ_N$, then all corners of $\xi^1_t$ give a contribution to $u(t)$ and from the assumption 
$\xi_1(t)\in \kA$ there are at least $k/10$ of them.

\medskip

If there are some contacts between $\xi^1_t$ and $\xi^2_t$, the idea of the proof is to control the contribution to the area and to $u(t)$ of each bubble.
Recall \eqref{defct} and \eqref{bubarea}.
Assume that the interval $[a,b]$ with $\#[a,b]\le N/k (\log k)^2$ is a bubble. It has at least one flippable corner.
For any $x\in [a,b]$, $i=1,2$ we have 

\begin{equation}
 \xi^1_t(a)-\frac{ k\#[a+1,x]}{N}  \le   \xi^i_t(x)\le  \xi^1_t(b)+\frac{ k\#[x+1,b]}{N},
\end{equation}
and hence
\begin{equation}
\max_{x\in [a,b]} (\xi^2_t-\xi^1_t)(x)\le (\xi^1_t(b)- \xi^1_t(a))+\frac{ k\#[a,b]}{N}
\end{equation}
From the definition of $\kA$, the right-hand side is smaller than $2(\log k)^4$ and hence 
\begin{equation}
A_{[a,b]}(t)\le  2(\log k)^4\#[a,b]\le  (N/k) (\log k)^6. 
\end{equation}
And hence (recall \eqref{defuab}, \eqref{bubarea})
\begin{equation}\label{smabub}
u_{[a,b]}(t)\ge 1\ge   \frac{A_{[a,b]}(t) k}{N (\log k)^6}. 
\end{equation}
For large bubbles with $\#[a,b]\ge \frac{2N} k (\log k)^2$
\begin{equation}
A_{[a,b]}(t)\le  H(t) \#[a,b]
\end{equation}
and thus from the definition of $\kA$
\begin{equation}\label{bigbub}
u_{[a,b]}(t) \ge \frac{k}{10N} \#[a,b]\ge \frac{A_{[a,b]}(t)k}{10N H(t)}. 
\end{equation}
The lemma is the proved by summing \eqref{smabub} and \eqref{bigbub} over all bubbles. 
\end{proof}

We are now ready to combine the ingredients for the the of Proposition \ref{nahnou} for arbitrary $k$.
We prove only \eqref{nahnou1}, as \eqref{nahnou2} can also be obtained in the same manner by adapting the technique used in the case $k=N/2$. 

\begin{proof}[Proof of \eqref{nahnou1} for arbitrary $k$]

We fix a constant $C_1(\gep)$  such that \eqref{cranchk} holds for $\gep/3$ instead of $\gep$.
$C_2(\gep,s)$ is chosen so that \eqref{crimouk} holds for $\gep/3$.
\begin{equation}
\begin{split}
\mathcal A&:=\left\{\int^T_0 \ind_{\{ \xi_1(t)\notin \mathfrak A \}}\dd t \le (T/k)\right\},\\
\mathcal B&:= \left\{ \forall i\in \{0,\dots, K\}, \int^{T}_0 \ind_{\{ H(t)\ge C_1 (i+1)^2 \sqrt{k}\}}\le (\alpha/4) (i+1)^{-2}T \right\},\\
\mathcal C&:=\left \{\forall i\in \{0,\dots, K\}, \cT_i\le C_2 N^2k 3^{-i}\right\} \cap \{\cT_\infty< C_2 N^2 \},
\end{split}
\end{equation}
 where 
\begin{equation}\begin{split}
T&:=40 N^2 C_1 C_2 \alpha^{-1} \max_{i\ge 0}\left[ (i+1)^4(2/3)^i\right],\\
T'&:=T+C_2N^2.
\end{split}\end{equation}
 From our definitions one has 
 \begin{equation}
 \bbP[\cA\cap \cB\cap\cC] \ge 1- \gep.
 \end{equation} 
 Then we conclude by doing the same reasoning as in the case $k=N/2$ that
 $$\cA\cap \cB\cap\cC \subset \{\tau \le T'\} $$
 using that for $k$ sufficiently large, on the event $\cA\cap \cB$ we have
 \begin{equation}
 \forall i\in \{0,\dots, K\}, \int^{T}_0 \ind_{\{ H(t)\ge C_1 (i+1)^2 \sqrt{k} \text{ or } \xi_1(t)\notin \mathfrak A \}}\le (\alpha/2) (i+1)^{-2}T 
 \end{equation}
 Concerning \eqref{nahnou2} we set $T=s^{-1}N^2$ 
 \begin{equation}
 \begin{split}
 \mathcal A'&:=\left\{\int^{T}_0 \ind_{\{ \xi_1(t)\notin \mathfrak A \}}\dd t \le (T/k)\right\},\\
 \mathcal B'&:= \left\{ \forall i\in \{0,\dots, K\}, \int^{N^2}_0 \ind_{\{ H(t)\ge s^{6}(i+1)^2 \sqrt{k}\}}\dd t\le (\alpha/4) (i+1)^{-2}N^2 \right\},\\
\mathcal C'&:=\left \{\forall i\in \{0,\dots, K\}, \cT_i\le s^{-7} N^2k 3^{-i}\right\} \cap \{\cT_\infty< s^{-1}N^2/2 \},
\end{split}
\end{equation}
 and use \eqref{cranchk2} , \eqref{crimouk2} and \eqref{fromzea2} to have for all sufficiently large $s$,
 \begin{equation}
  \bbP[\cA'\cap \cB'\cap\cC'] \ge e^{-Cs^{-6}} 
 \end{equation}
 and conclude as in the proof for $k=N/2$ that 
  $$ \{\tau \le 2N^2\}\subset \cA'\cap \cB'\cap\cC' ,$$
 using that if $N$ (and thus $k$) is sufficiently large we have on the event $\cA'\cap\cB'$
  \begin{equation}
 \forall i\in \{0,\dots, K\}, \int^{N^2}_0 \ind_{\{ H(t)\ge s^{6} (i+1)^2 \sqrt{k} \text{ or }  \xi_1(t)\notin \mathfrak A \}}\dd t\le (\alpha/4) (i+1)^{-2}N^2.
 \end{equation}
 \end{proof}

 {\bf Acknowledgement: } The author is grateful to Fran\c{c}ois Huyveneers and Simenhaus  for enlightening discussions and in particular for bringing to his knowledge the existence of the comparison inequalities for the exclusion process  of Lemma \ref{ligett}.  
 
 \appendix

\section{Estimate for the discrete heat equation: proof of Lemma \ref{bound}}

Let 
$p_t(x,y)$ denote the heat kernel of on the discrete circle ( $p_t(x,\cdot)$ corresponds to the probability distribution of 
a simple random walk starting from $x$ at time $t$).
For fixed $x\in \bbZ_N$, the function 
$$u(y,t):=p_t(x,y)$$ 
is the solution of the discrete heat equation on $\bbZ_N$
\begin{equation}\label{disheat}
\begin{split}
\partial_t u(y,t)&=u(y+1,t)+u(y-1,t)-2u(y,t),\\
u(y,0)&:=\ind_{\{ y=x\} }.
\end{split}
\end{equation}
Without loss of generality one restrict to the case $x=0$.
The solution of the above equation can be found by a decomposition on a base of eigenvalues the discrete Laplacian.
We set for all $i\in\{1,\dots, \lfloor N/2\rfloor\}$ 
$$\gl_i:=2\left(1-\cos\big(\frac{2\pi i}{N}\big)\right),$$
to be the eigenvalue associated to the normalized eigenfunction,
$$y\mapsto \sqrt{2/N}\cos(2i\pi y/N)$$,
the factor in front being $\sqrt{1/N}$ instead of $\sqrt{2/N}$ for $i=N/2$.
We have to add also sine eigenfunctions to have a base but the projection of $\ind_{\{ y=0\}}$ on these eigenfunction is equal to zero.
By Fourier decomposition, we have
\begin{multline}
\left| u(y,t)-\frac{1}{N}\right|=\left| \frac{1}{N}\sum_{i=1}^{\lfloor N/2\rfloor}(2-\ind_{\{i=N/2\}}) e^{-\gl_i t}\cos(2i\pi y/N)\right|\\
\le \frac{2}{N} \sum_{i=1}^{\lfloor N/2\rfloor}e^{-\gl_i t}
\le \frac{2}{N}\frac{1}{e^{\gl_1 t}-1}.
\end{multline}
where in the last inequality we just used $e^{-\gl_i t}\le e^{-i\gl_1 t}$.
When $t\ge N^2$, we have $e^{\gl_1 t}\ge 3$ and hence we have

\begin{equation}\begin{split}
u(y,t)&\le \frac{2}{N},\\
\left| u(y,t)-\frac{1}{N}\right|&\le \frac{4}{N}e^{-\gl_1 t}.
\end{split}
\end{equation}
Both statements in Lemma \ref{bound} are derived from these inequalities. For the first one we need to use that 
$\bbE[\eta_t(x)]$ is a solution of the discrete heat equation. 

\qed

\end{document}